\documentclass[a4paper,12pt]{article}

\usepackage{parskip}
\usepackage{fullpage}
\usepackage{abstract}
\usepackage{cancel}

\usepackage{titlesec}
\def\sectionfont{\sffamily\Large\bfseries\boldmath}
\def\subsectionfont{\sffamily\large\bfseries\boldmath}
\def\paragraphfont{\sffamily\normalsize\bfseries\boldmath}
\titleformat*{\section}{\sectionfont}
\titleformat*{\subsection}{\subsectionfont}
\titleformat*{\subsubsection}{\paragraphfont}
\titleformat*{\paragraph}{\paragraphfont}
\titleformat*{\subparagraph}{\paragraphfont}

\usepackage{amsmath,amssymb,mathtools,amsfonts}
\usepackage[english]{babel}
\usepackage[titletoc,title]{appendix}
\usepackage{enumitem}
\setlist{nolistsep}
\usepackage{hyperref}
\hypersetup{%
  colorlinks=true,
  linktoc=all,
  linkcolor=black,
  citecolor= black,
  urlcolor=blue
}

\title{\bfseries\sffamily
  Homogeneous Self-Dual Embedding via Perspective Functions
}

\author{%
  Goran Banjac \\
  \href{mailto:goranbanjac1989@gmail.com}{\texttt{goranbanjac1989@gmail.com}}
}

\usepackage{amsthm}

\newtheoremstyle{exampstyle}
  {.5\baselineskip}
  {\topsep}
  {}
  {}
  {\bfseries}
  {.}
  {.5em}
  {}
\theoremstyle{exampstyle}

\newtheorem{theorem}{Theorem}[section]
\newtheorem{corollary}[theorem]{Corollary}

\newtheorem{lemma}[theorem]{Lemma}
\newtheorem{proposition}[theorem]{Proposition}
\newtheorem{remark}[theorem]{Remark}

\newcommand{\eg}{\emph{e.g.},\ }
\newcommand{\ie}{\emph{i.e.},\ }

\newcommand{\mcf}{\mathcal}

\newcommand{\mbb}{\mathbb}

\renewcommand{\Re}{\mbb{R}}
\newcommand{\Nat}{\mbb{N}}
\newcommand{\symm}{\mbb{S}}

\newcommand{\eqdef}{\coloneqq}

\newcommand{\Alg}{Alg.}
\newcommand{\Prop}{Prop.}

\newcommand{\Thm}{Thm.}
\newcommand{\Lem}{Lem.}

\newcommand{\App}{App.}

\DeclareMathOperator*{\argmin}{\operatorname{argmin}}

\DeclareMathOperator{\dom}{dom}

\DeclareMathOperator{\prox}{Prox}
\DeclareMathOperator{\range}{ran}

\DeclareMathOperator{\epi}{epi}

\newcommand{\eps}{\varepsilon}

\newcommand{\indicator}[1]{\iota_{#1}}
\newcommand{\project}[1]{P_{#1}}

\newcommand{\polar}[1]{{#1}^\ominus}
\newcommand{\dual}[1]{{#1}^\oplus}

\newcommand{\recession}[1]{{\rm rec}\,{#1}}
\newcommand{\support}[1]{\sigma_{#1}}
\newcommand{\closure}[1]{\overline{#1}\,}
\newcommand{\inv}{^{-1}}
\newcommand{\norm}[1]{\lVert#1\rVert}
\newcommand{\half}{\tfrac{1}{2}}
\newcommand{\innerprod}[2]{\left\langle{#1}\mid{#2}\right\rangle}

\newcommand{\seq}[1]{({#1})_{k\in\Nat}}
\newcommand{\setB}{B}
\newcommand{\setC}{C}
\newcommand{\setD}{D}

\usepackage{algorithm}
\usepackage[noend]{algpseudocode}

\newcounter{algorithmctr}[section]
\renewcommand{\thealgorithmctr}{\thesection.\arabic{algorithmctr}}

\usepackage{tikz}
\usepackage{pgfplots}
\pgfplotsset{compat=newest} 

\begin{document}

\maketitle

\begin{abstract}
  We present a generalization of the well-known homogeneous self-dual embedding model, which is widely used in conic optimization.
  The new embedding applies to a problem of minimizing the sum of two proper lower-semicontinuous convex functions and can be represented as a single inequality that uses perspectives of these functions and of their conjugates.
  A solution to the proposed embedding encodes a primal-dual solution to the original problem when available, or an infeasibility certificate otherwise.
  We then use the Douglas-Rachford algorithm to find a solution to the embedding and discuss its efficient implementation by exploiting the problem structure.
  The resulting algorithm recovers an existing method for solving quadratic cone programs as a special case.
  We demonstrate the generality and effectiveness of the algorithm on a class of convex optimization problems with non-smooth objective function and non-conic constraints.
\end{abstract}

\section{Introduction}\label{sec:intro}

Solving a convex optimization problem by first reducing it to an equivalent \emph{linear cone program (LCP)}, \ie a problem of minimizing a linear objective function subject to a conic constraint, has been a standard strategy for several decades.
There are multiple reasons for using this canonical form.
First, it is due to its expressive power, \ie the ability to reduce a wide range of convex optimization problems to an LCP using only a handful of cone types.
This \emph{canonicalization} procedure is easy to automate and is implemented in optimization parsers such as CVXPY \cite{Diamond:2016} and JuMP \cite{Dunning:2017}.
Having such a canonical form also simplifies the development of general-purpose numerical solvers.
Second, it is due to the development of interior-point methods for LCPs \cite{Nesterov:1994,Nemirovski:2008}, which are the methods of choice for solving small and medium-size problem instances.
Third, other algorithmic advances such as the \emph{homogeneous self-dual embedding (HSDE)} model, originally developed for linear programs \cite{Goldman:1956,Ye:1994}, also extend to LCPs \cite{Luo:2000}.
This embedding model encodes both optimality and infeasibility conditions into a single convex feasibility problem, and is at the heart of many interior-point solvers such as MOSEK \cite{mosek}, CVXOPT \cite{cvxopt}, and ECOS \cite{Domahidi:2013}, but also of an operator splitting solver SCS \cite{O'Donoghue:2016}.

Nevertheless, there has been an increased interest in extending the canonical form so that optimization problem terms that appear frequently in real-world applications are captured in a more direct way.
The main reason behind this is to solve optimization problems more efficiently.
Solving a problem in its original form is often more efficient as the canonicalization procedure typically introduces additional variables and constraints, and can destroy structure present in the original problem.
For instance, the ubiquity of a quadratic objective function in real-world applications has motivated the development of COSMO \cite{Garstka:2021}, an operator splitting solver for \emph{quadratic cone programs (QCPs)}, that handles quadratic objective terms directly rather than relying on reductions to second-order cone constraints.
COSMO is not built on top of an HSDE model, but instead relies on the alternating direction method of multipliers, whose iterates yield conclusive information regarding problem infeasibility~\cite{Banjac:2019}.
This new canonical form was also adopted in \cite{O'Donoghue:2021}, which extends the SCS solver, originally developed for LCPs \cite{O'Donoghue:2016}, to handle quadratic objective terms directly.
This extension formulates a QCP as a linear complementarity problem and applies the Douglas-Rachford algorithm (DRA) to its HSDE model proposed in \cite{Andersen:1999}.
The same HSDE model was used in \cite{Goulart:2026} to develop Clarabel, an interior-point solver for QCPs.

Another example of a ubiquitous term in real-world optimization problems are box constraints, \ie simultaneous lower and upper bounds on an optimization variable.
While some numerical solvers handle such constraints directly, such as OSQP \cite{Stellato:2020} and PDLP \cite{Applegate:2026}, and rely on using algorithm iterates to construct an infeasibility certificate \cite{Banjac:2019,Applegate:2024}, such non-conic constraints cannot be embedded directly within current HSDE frameworks, which rely instead on their conic reductions.

In this paper we propose an HSDE model that applies to a problem of minimizing the sum of two proper lower-semicontinuous convex functions and uses perspective functions to embed optimality and infeasibility conditions into a single inequality.
This new model extends the applicability of the HSDE framework beyond QCPs and opens a door to new algorithmic developments.
We then apply DRA to the proposed embedding.
Each iteration of the algorithm evaluates projections onto epigraphs of some particular perspective functions.
We then show how exploiting the problem structure yields a simpler DRA iteration, \ie each projection operation boils down to solving a scalar equation and computing the proximity operator of a function used in the original problem description.
Interestingly, when applied to QCPs, the proposed algorithm recovers the method underpinning the SCS solver \cite{O'Donoghue:2021}.
To demonstrate the applicability of the new HSDE model beyond QCPs, we apply the algorithm to a class of convex optimization problems with a non-smooth objective function and box constraints.

A non-directly related use of perspective functions within the HSDE framework was reported in \cite{Zhang:2004}, which considers LCPs with additional constraints of the form
\[
  f_i(x) \le 0, \quad i=1,\ldots,m,
\]
where each $f_i$ is a smooth convex function.
Each such constraint is then reformulated as
\begin{align*}
  & p f_i(x/p) \le q, \quad p \ge 0 \\
  & p = 1, \quad q = 0,
\end{align*}
and, noting that the inequalities above can be represented as $(x,p,q)\in\mcf{K}_i$, where $\mcf{K}_i$ is a nonempty closed convex cone, the original problem can be reduced to a conic program.
However, this approach is not very different from the standard canonicalization procedure that reformulates a convex optimization problem into an equivalent LCP.
Note that a similar reformulation is used in \cite[\S6.2]{O'Donoghue:2021} to reduce box constraints to the intersection of a cone and an affine equality constraint.

The paper is organized as follows.
We introduce some definitions and notation in the remainder of \S\ref{sec:intro}, and some known results on perspective functions in \S\ref{sec:perspective}.
\S\ref{sec:hsde} introduces a new HSDE model and derives some of its properties.
\S\ref{sec:dra} applies DRA to the new embedding and discusses its implementation and convergence.
\S\ref{sec:qcp} applies the proposed algorithm to QCPs, while \S\ref{sec:structured-cp} applies it to a class of non-conic optimization problems.
Finally, \S\ref{sec:numerics} demonstrates the effectiveness of the proposed algorithm on several small numerical examples.

\subsection{Notation}\label{sec:notation}
All definitions introduced here are standard and can be found in \cite{Bauschke:2017:book}, to which we also refer for basic results on convex analysis.

Let $\Nat$ denote the set of positive integers, and $\mcf{H}$ be a real Hilbert space with inner product $\innerprod{\cdot}{\cdot}$ and induced norm $\norm{\,\cdot\,}$.
Let $\setC$ be a nonempty subset of $\mcf{H}$ with $\closure{\setC}$ being its \emph{closure}.
The \emph{kernel} of a linear operator $A$ is denoted by $\ker{A}$ and its \emph{range} by $\range{A}$.

The set of proper lower semicontinuous convex functions from $\mcf{H}$ to $\left]-\infty,+\infty\right]$ is denoted by $\Gamma_0(\mcf{H})$.
For a function $f\in\Gamma_0(\mcf{H})$, we define its:
\begin{align*}
  &\text{\emph{domain}:} && \dom f = \lbrace x\in\mcf{H} \mid f(x) < +\infty \rbrace, \\
  &\text{\emph{epigraph}:} && \epi f = \lbrace (x,t)\in\mcf{H}\times\Re \mid f(x) \le t \rbrace, \\
  &\text{\emph{Fenchel conjugate}:} && f^* \colon \mcf{H}\to\left]-\infty,+\infty\right] \colon u\mapsto\sup_{x\in\mcf{H}}\left( \innerprod{x}{u} - f(x) \right), \\
  &\text{\emph{recession function}:} && \recession{f} \colon \mcf{H}\to\left]-\infty,+\infty\right] \colon y\mapsto\sup_{x\in\dom{f}} \left( f(x+y)-f(x) \right), \\
  &\text{\emph{perspective function}:} && \widetilde{f} \colon \mcf{H} \times \Re \to \left]-\infty,+\infty\right] \colon (x,\tau) \mapsto \begin{cases}
    \tau f(x/\tau) & \tau > 0 \\
    (\recession{f})(x) & \tau=0 \\
    +\infty & \text{otherwise,}
  \end{cases} \\
    &\text{\emph{proximity operator}:} && \prox_f \colon \mcf{H}\to\mcf{H} \colon x\mapsto\argmin_{y\in\mcf{H}}\left( f(y) + \half \norm{y-x}^2 \right).
\end{align*}
For a nonempty closed convex set $\setC\subseteq\mcf{H}$, we define its:
\begin{align*}
  &\text{\emph{polar cone}:} && \polar{\setC} = \Big\lbrace u\in\mcf{H} \mid \sup_{x\in\setC}\innerprod{x}{u} \le 0 \Big\rbrace, \\
  &\text{\emph{recession cone}:} && \recession{\setC} = \left\lbrace x\in\mcf{H} \mid (\forall y\in\setC) \: x+y\in\setC \right\rbrace, \\
  &\text{\emph{indicator function}:} && \indicator{\setC} \colon \mcf{H}\to\left[0,+\infty\right] \colon x\mapsto\begin{cases} 0 & x\in\setC \\ +\infty & \text{otherwise,}\end{cases} \\
  &\text{\emph{support function}:} && \support{\setC} \colon \mcf{H}\to\left]-\infty,+\infty\right] \colon u\mapsto\sup_{x\in\setC} \innerprod{x}{u}, \\
  &\text{\emph{projection operator}:} && \project{\setC} \colon \mcf{H}\to\mcf{H} \colon x\mapsto\argmin_{y\in\setC}\,\norm{y-x}.
\end{align*}
The dual cone of $\setC$ is $\dual{\setC}=-\polar{\setC}$.

We denote an $n$-dimensional box with lower bound $l$ and upper bound $u$ by
\[
  [l,u] \eqdef \lbrace x\in\Re^n \mid l \le x \le u \rbrace,
\]
where $l_i\in\Re\cup\{-\infty\}$ and $u_i\in\Re\cup\{+\infty\}$ for $i=1,\ldots,n$.
A set of $n$-dimensional vectors is denoted by $\Re^n$, a set of nonnegative $n$-dimensional vectors by $\Re_+^n$, a set of $m$-by-$n$ matrices by $\Re^{m\times n}$, and a set of $n$-by-$n$ symmetric positive semidefinite matrices  by $\symm_+^n$.

\section{Preliminaries on Perspective Functions}\label{sec:perspective}

The perspective of $f\in\Gamma_0(\mcf{H})$, denoted by $\widetilde{f}$ and defined in \S\ref{sec:notation}, is a function in $\Gamma_0(\mcf{H}\times\Re)$ \cite[\Prop~9.42]{Bauschke:2017:book}.
Since $\widetilde{f}$ is positively homogeneous, it is easy to show that its epigraph is a nonempty closed convex cone.
The following proposition shows how we can characterize the polar of $\epi \widetilde{f}$ via the perspective of $f^*$.

\begin{proposition}[{\cite[\Thm~14.4]{Rockafellar:1970}}]\label{prop:epi-perspective}
  Let $f\in\Gamma_0(\mcf{H})$ and consider a nonempty closed convex cone defined as the epigraph of its perspective, \ie
  \[
    \epi \widetilde{f} = \lbrace (x,\tau,\kappa)\in\mcf{H}\times\Re^2 \mid \widetilde{f}(x,\tau) \le \kappa \rbrace.
  \]
  Then the polar of $\epi \widetilde{f}$ is given by
  \[
    \polar{(\epi \widetilde{f})} = \lbrace (\lambda,\eps,\delta)\in\mcf{H}\times\Re^2 \mid \widetilde{f^*}(\lambda,-\delta) \le -\eps \rbrace.
  \]
\end{proposition}

A systematic study of perspective function properties is conducted in \cite{Combettes:2018}, including characterizations of its Fenchel conjugate and subdifferential, while characterizations and evaluation of its proximity operator are studied in \cite{Combettes:2018b,Briceno-Arias:2024b,Briceno-Arias:2024c,Briceno-Arias:2024}.
The following two propositions show how we can express the projection onto $\epi \widetilde{f}$ in terms of the proximity operator of $\widetilde{f}$, as well as how to express the proximity operator of $\widetilde{f}$ in terms of the proximity operator of $f$.

\begin{proposition}[{\cite[\Thm~3.1]{Briceno-Arias:2024c}}]\label{prop:project-Kf}
  Let $f\in\Gamma_0(\mcf{H})$ and let $(\bar{x},\bar{\tau},\bar{\kappa})\in\mcf{H}\times\Re^2$.
  Then we have
  \[
    \project{\epi\widetilde{f}}(\bar{x},\bar{\tau},\bar{\kappa}) =
    \begin{cases}
      ( \project{\closure{\dom}\widetilde{f}}(\bar{x},\bar{\tau}), \bar{\kappa} ) & \tilde{f}(\project{\closure{\dom}\widetilde{f}}(\bar{x},\bar{\tau})) \le \bar{\kappa} \\
      ( \prox_{\mu\widetilde{f}}(\bar{x},\bar{\tau}), \bar{\kappa} + \mu ) & \text{otherwise},
    \end{cases}
  \]
  where $\mu\in ] 0, -\bar{\kappa} + \tilde{f}(\project{\closure{\dom}\widetilde{f}}(\bar{x},\bar{\tau})) ]$ is the unique solution to
  \[
    \mu + \bar{\kappa} - \widetilde{f}( \prox_{\mu\widetilde{f}}(\bar{x},\bar{\tau}) ) = 0.
  \]
\end{proposition}

\begin{proposition}[{\cite[\Prop~2.1]{Briceno-Arias:2024c}}]\label{prop:prox-ft}
  Let $f\in\Gamma_0(\mcf{H})$, let $\mu>0$, and let $(\bar{x},\bar{\tau})\in\mcf{H}\times\Re$.
  Then we have
  \[
    \prox_{\mu \widetilde{f}}(\bar{x},\bar{\tau}) = \begin{cases}
      \big( \bar{x} - \mu \project{\closure{\dom} f^*}(\bar{x}/\mu), 0 \big) & \bar{\tau} + \mu f^*(\project{\closure{\dom} f^*}(\bar{x}/\mu)) \le 0 \\
      \big( \tau \prox_{(\mu/\tau)f}(\bar{x}/\tau) , \tau \big) & \text{otherwise},
    \end{cases}
  \]
  where $\tau\in ]0, \bar{\tau} + \mu f^*(\project{\closure{\dom} f^*}(\bar{x}/\mu))]$ is the unique solution to
    \[
      \tau = \bar{\tau} + \mu f^* \big( \prox_{(\tau/\mu)f^*}(\bar{x}/\mu) \big).
    \]
\end{proposition}

In general, projecting onto $\epi \widetilde{f}$ requires solving two coupled scalar equations to find $\tau$ and $\mu$.
Nevertheless, we will show in \S\ref{subsec:dra:projF} that evaluating the projection onto $\epi \widetilde{f}$ can be simplified for some particular functions and arguments that arise in our application.

\section{Homogeneous Self-Dual Embedding}\label{sec:hsde}

Consider the following convex optimization problem:
\begin{equation}\label{eqn:primal}
  \underset{x\in\mcf{H}}{\rm minimize} \quad f(x) + g(x),
  \tag{\mbox{$\mcf{P}$}}
\end{equation}
where $f$ and $g$ are functions in $\Gamma_0(\mcf{H})$.
The Fenchel dual of \eqref{eqn:primal} takes the following form:
\begin{equation}\label{eqn:dual}
  \underset{\nu\in\mcf{H}}{\rm maximize} \quad -f^*(\nu) - g^*(-\nu).
  \tag{\mbox{$\mcf{D}$}}
\end{equation}
We assume that strong duality holds, \ie the optimal values of problems \eqref{eqn:primal} and \eqref{eqn:dual} coincide.
To simplify the notation in subsequent analysis, we define $F\colon\mcf{H}\times\mcf{H}\to\left]-\infty,+\infty\right]$ and $G\colon\mcf{H}\times\mcf{H}\to\left]-\infty,+\infty\right]$ as
\begin{align*}
  F(x,\nu) &\coloneqq f(x) + f^*(\nu) \\
  G(x,\nu) &\coloneqq g(x) + g^*(-\nu).
\end{align*}
Note that $F$ and $G$ are functions in $\Gamma_0(\mcf{H}\times\mcf{H})$.
Our goal is to find a primal-dual solution to the problem pair \eqref{eqn:primal}-\eqref{eqn:dual}, or a certificate of primal or dual infeasibility.
To that end, we are interested in finding a nonzero tuple $(x,\nu,\tau) \in \mcf{H} \times \mcf{H} \times \Re_+$ such that the following inequality holds:
\begin{equation}\label{eqn:inequality}
  \widetilde{F}(x,\nu,\tau) + \widetilde{G}(x,\nu,\tau) \le 0,
  \tag{\mbox{$\mcf{I}$}}
\end{equation}
which can be formulated as the following feasibility problem:
\begin{equation}\label{eqn:hsde}
  \begin{array}{ll}
    {\rm find} & (x,\nu,\tau,\kappa) \\
    \textrm{s.\ t.} & \widetilde{F}(x,\nu,\tau) \le -\kappa \\
    & \widetilde{G}(x,\nu,\tau) \le \kappa,
  \end{array}
  \tag{HSDE}
\end{equation}
whose properties we study in the remainder of this section.

\subsection{Convexity and homogeneity of the solution set}\label{sec:hsde:sol}

The set of solutions $(x,\nu,\tau,\kappa)$ to \eqref{eqn:hsde} is a convex cone as it is the intersection of the following nonempty closed convex cones:
\begin{subequations}\label{eqn:KF-KG}
\begin{align}
  \mcf{F} &\eqdef
  \lbrace (x,\nu,\tau,\kappa) \in \mcf{H}\times\mcf{H}\times\Re^2 \mid \widetilde{F}(x,\nu,\tau) \le -\kappa \rbrace \label{eqn:KF-KG:KF} \\
  \mcf{G} &\eqdef
  \lbrace (x,\nu,\tau,\kappa) \in \mcf{H}\times\mcf{H}\times\Re^2 \mid \widetilde{G}(x,\nu,\tau) \le \kappa \rbrace. \label{eqn:KF-KG:KG}
\end{align}
\end{subequations}
We say that problem \eqref{eqn:hsde} is homogeneous because if $(x,\nu,\tau,\kappa)$ is its solution, then so is $(tx,t\nu,t\tau,t\kappa)$ for any $t\ge0$.

\subsection{Encoding zero duality gap}

A tuple $(x,\nu,\tau)$ satisfying \eqref{eqn:inequality} with $\tau>0$ encodes optimality of the primal-dual problem pair \eqref{eqn:primal}-\eqref{eqn:dual}.
In particular, $x/\tau$ is a primal and $\nu/\tau$ is a dual solution.
To see this, note that dividing \eqref{eqn:inequality} by $\tau>0$ yields
\begin{equation}\label{eqn:zero-duality-gap}
  f(x/\tau) + g(x/\tau) + f^*(\nu/\tau) + g^*(-\nu/\tau) \le 0.
\end{equation}
Since the left-hand side of \eqref{eqn:zero-duality-gap} is exactly the duality gap, which must be greater than or equal to zero (due to weak duality), it follows that the above inequality must be satisfied with equality.
Moreover, the fact that the left-hand side of the inequality is finite means that $(x/\tau) \in \dom f \cap \dom g$ and $(\nu/\tau) \in \dom f^* \cap (-\dom g^*)$.
Hence, the primal feasibility of $x/\tau$, the dual feasibility of $\nu/\tau$, and the zero duality gap imply that $(x/\tau,\nu/\tau)$ is a primal-dual solution.

\subsection{Encoding strong infeasibility}

A tuple $(x,\nu,\tau)$ satisfying \eqref{eqn:inequality} with strict inequality encodes strong infeasibility of \eqref{eqn:primal} and/or \eqref{eqn:dual}.
Due to weak duality, it must be that $\tau=0$, which reduces \eqref{eqn:inequality} to
\[
  (\recession{f})(x) + (\recession{g})(x) + (\recession{f^*})(\nu) + (\recession{g^*})(-\nu) < 0.
\]
Hence, it must be that at least one of the following inequalities holds:
\begin{subequations}\label{eqn:infeas-strict-ineq}
\begin{align}
  (\recession{f^*})(\nu) + (\recession{g^*})(-\nu) &< 0 \label{eqn:infeas-strict-ineq:f*} \\
  (\recession{f})(x) + (\recession{g})(x) &< 0, \label{eqn:infeas-strict-ineq:g*}
\end{align}
\end{subequations}
which encode strong infeasibility of \eqref{eqn:primal} and \eqref{eqn:dual}, respectively.
To see this, note that \eqref{eqn:infeas-strict-ineq:f*} can be written as
\[
  0 > (\recession{f^*})(\nu) + (\recession{g^*})(-\nu) = \support{\dom{f}}(\nu) + \support{\dom{g}}(-\nu),
\]
or equivalently,
\[
  \sup_{x\in\dom f}\innerprod{x}{\nu} < \inf_{z\in\dom g}\innerprod{z}{\nu},
\]
which means that there exists a hyperplane (with normal $\nu$) separating $\dom{f}$ and $\dom{g}$ strongly, implying that the sets do not intersect.

Similarly, we can write
\[
  0 > (\recession{f})(x) + (\recession{g})(x) = \support{\dom{f^*}}(x) + \support{\dom{g^*}}(x),
\]
or equivalently
\[
  \sup_{\nu\in\dom{f^*}}\innerprod{\nu}{x} < \inf_{\lambda\in -\dom g^*}\innerprod{\lambda}{x},
\]
which means that there exists a hyperplane (with normal $x$) separating $\dom{f^*}$ and $-\dom{g^*}$ strongly.

\subsection{Non-conclusive case}

A nonzero tuple $(x,\nu,0)$ satisfying \eqref{eqn:inequality} with equality can still serve as an infeasibility certificate if one of the strict inequalities in \eqref{eqn:infeas-strict-ineq} holds.
Otherwise, nothing can be concluded about the original problem.
We summarize all the discussed cases in Table~\ref{tab:solutions}.
Note also that, since the solution set is the intersection of nonempty closed convex cones $\mcf{F}$ and $\mcf{G}$, zero is always a solution to \eqref{eqn:inequality}.

\begin{table}[t]
\centering
\begin{tabular}{c | c c } 
  & $\tau > 0$ & $\tau = 0$ \\[.5ex]
  \hline \\[-1.5ex]
  LHS of \eqref{eqn:inequality} $< 0$ & N/A & infeasible \\[.5ex]
  LHS of \eqref{eqn:inequality} $= 0$ & solved & non-conclusive
\end{tabular}
\caption{How the solutions to \eqref{eqn:inequality} relate to the status of problem pair \eqref{eqn:primal}-\eqref{eqn:dual}.}
\label{tab:solutions}
\end{table}

\subsection{Self-duality of HSDE}

Problem~\eqref{eqn:hsde} is equivalent to the following one:
\begin{equation}\label{eqn:hsde-composite}
  \underset{(x,\nu,\tau,\kappa)}{\rm minimize} \quad \indicator{\mcf{F}}(x,\nu,\tau,\kappa) + \indicator{\mcf{G}}(x,\nu,\tau,\kappa).
\end{equation}
The dual problem then takes the following form:
\[
  \underset{(\lambda,z,\eps,\delta)}{\rm maximize} \quad -\indicator{\polar{\mcf{F}}}(\lambda,z,\eps,\delta) - \indicator{\polar{\mcf{G}}}(-\lambda,-z,-\eps,-\delta),
\]
where, due to Proposition~\ref{prop:epi-perspective} and Lemma~\ref{lem:FG_conjugate}, the polar cones are given by
\begin{subequations}\label{KF-KG-polar}
\begin{align}
\begin{split}
  \polar{\mcf{F}} &= \lbrace (\lambda,z,\eps,\delta) \in \mcf{H}\times\mcf{H}\times\Re^2 \mid \widetilde{F^*}(\lambda,z,\delta) \le -\eps \rbrace \\
  & = \lbrace (\lambda,z,\eps,\delta) \in \mcf{H}\times\mcf{H}\times\Re^2 \mid \widetilde{F}(z,\lambda,\delta) \le -\eps \rbrace
\end{split} \label{KF-KG-polar:KF} \\
\begin{split}
  \polar{\mcf{G}} &= \lbrace (\lambda,z,\eps,\delta) \in \mcf{H}\times\mcf{H}\times\Re^2 \mid \widetilde{G^*}(\lambda,z,-\delta) \le -\eps \rbrace \\
  &= \lbrace (\lambda,z,\eps,\delta) \in \mcf{H}\times\mcf{H}\times\Re^2 \mid \widetilde{G}(-z,-\lambda,-\delta) \le -\eps \rbrace.
\end{split} \label{KF-KG-polar:KG}
\end{align}
\end{subequations}
Hence, the dual problem can be written as
\[
  \begin{array}{ll}
    {\rm find} & (z,\lambda,\delta,\eps) \\
    \textrm{s.\ t.} & \widetilde{F}(z,\lambda,\delta) \le -\eps \\
    & \widetilde{G}(z,\lambda,\delta) \le \eps,
  \end{array}
\]
which is equivalent to \eqref{eqn:hsde}.
Hence, the problem is self-dual.

\section{Douglas-Rachford Algorithm}\label{sec:dra}

The Douglas-Rachford algorithm (DRA) is a well-known operator splitting method that can be used to solve problems of the form~\eqref{eqn:primal} \cite{Lions:1979}.
Starting from some $s^1\in\mcf{H}$, the algorithm generates the following iterates:
\begin{align*}
  x^k &\gets \prox_{\gamma f} (s^k) \\
  \tilde{x}^k &\gets \prox_{\gamma g}(2 x^k - s^k) \\
  s^{k+1} &\gets s^k + \tilde{x}^k - x^k,
\end{align*}
where $\gamma>0$ is the algorithm parameter.

In this section we analyze DRA when applied to problem~\eqref{eqn:hsde-composite}, \ie
\begin{subequations}\label{eqn:dra-hsde}
\begin{align}
  (x^k,\nu^k,\tau^k,\kappa^k) &\gets \project{\mcf{F}} (s_x^k,s_\nu^k,s_\tau^k,s_\kappa^k) \label{eqn:dra-hsde:PF} \\
  (\tilde{x}^k,\tilde{\nu}^k,\tilde{\tau}^k,\tilde{\kappa}^k) &\gets \project{\mcf{G}} \big( 2(x^k,\nu^k,\tau^k,\kappa^k) - (s_x^k,s_\nu^k,s_\tau^k,s_\kappa^k) \big) \label{eqn:dra-hsde:PG} \\
  (s_x^{k+1},s_\nu^{k+1},s_\tau^{k+1},s_\kappa^{k+1}) &\gets (s_x^k,s_\nu^k,s_\tau^k,s_\kappa^k) + (\tilde{x}^k,\tilde{\nu}^k,\tilde{\tau}^k,\tilde{\kappa}^k) - (x^k,\nu^k,\tau^k,\kappa^k) \label{eqn:dra-hsde:s-update}.
\end{align}
\end{subequations}
Note that the algorithm parameter $\gamma$ plays no role in evaluating proximity operators of indicator functions as they reduce to projection operators for any value of $\gamma>0$.

\subsection{Simplifying the DRA iteration}

We next show how we can simplify the DRA steps by exploiting structure of problem~\eqref{eqn:hsde-composite}.
Assuming
\begin{equation}\label{eqn:x-nu-init}
  s_x^k=s_\nu^k=p^k
  \quad\text{and}\quad
  s_\tau^k=s_\kappa^k=r^k,
\end{equation}
for some $p^k\in\mcf{H}$ and $r^k\in\Re$, we can represent \eqref{eqn:dra-hsde:PF} as
\[
  (x^k,\nu^k,\tau^k,\kappa^k) = \project{\mcf{F}}(p^k,p^k,r^k,r^k).
\]
Now consider the projection of the same point $(p^k,p^k,r^k,r^k)$ onto the polar of $\mcf{F}$, \ie
\[
  (\lambda^k,z^k,\eps^k,\delta^k) = \project{\polar{\mcf{F}}}(p^k,p^k,r^k,r^k).
\]
Observing from \eqref{eqn:KF-KG:KF} and \eqref{KF-KG-polar:KF} that the sets $\mcf{F}$ and $\polar{\mcf{F}}$ are equivalent up to the permutation of arguments, it follows from \eqref{eqn:x-nu-init} and \cite[\Prop~29.2(ii)]{Bauschke:2017:book} that
\begin{equation}\label{eqn:zmde-via-xntk}
  x^k = z^k, \quad
  \nu^k = \lambda^k, \quad
  \tau^k = \delta^k, \quad
  \kappa^k = \eps^k.
\end{equation}
Due to Moreau's decomposition \cite[\Thm~6.30]{Bauschke:2017:book}, we also have
\[
  x^k + \lambda^k = p^k, \quad
  \nu^k + z^k = p^k, \quad
  \tau^k + \eps^k = r^k, \quad
  \kappa^k + \delta^k = r^k,
\]
and thus
\begin{equation}\label{eqn:nu_from_x}
  \nu^k = p^k - x^k, \quad
  \kappa^k = r^k - \tau^k.
\end{equation}
It also follows from Moreau's decomposition that
\[
  \innerprod{(x^k,\nu^k,\tau^k,\kappa^k)}{(\lambda^k,z^k,\eps^k,\delta^k)} = 0,
\]
which due to \eqref{eqn:zmde-via-xntk} and \eqref{eqn:nu_from_x} reduces to
\begin{equation}\label{x-tau-orth}
  \innerprod{x^k}{p^k-x^k} + \tau^k (r^k - \tau^k) = 0.
\end{equation}
Observing that the arguments of $\project{\mcf{G}}$ in \eqref{eqn:dra-hsde:PG} satisfy
\begin{align*}
  2x^k - s_x^k &= 2x^k - p^k = x^k - \nu^k \eqqcolon \tilde{p}^k \\
  2\nu^k - s_\nu^k &= 2(p^k - x^k) - p^k = -\tilde{p}^k \\
  2\tau^k - s_\tau^k &= 2\tau^k - r^k = \tau^k - \kappa^k \eqqcolon \tilde{r}^k \\
  2\kappa^k - s_\kappa^k &= 2(r^k - \tau^k) - r^k = -\tilde{r}^k,
\end{align*}
we can represent \eqref{eqn:dra-hsde:PG} as
\[
  (\tilde{x}^k,\tilde{\nu}^k,\tilde{\tau}^k,\tilde{\kappa}^k) = \project{\mcf{G}}(\tilde{p}^k,-\tilde{p}^k,\tilde{r}^k,-\tilde{r}^k).
\]
Now consider the projection of the same point $(\tilde{p}^k,-\tilde{p}^k,\tilde{r}^k,-\tilde{r}^k)$ onto the polar of $\mcf{G}$, \ie
\[
  (\tilde{\lambda}^k,\tilde{z}^k,\tilde{\eps}^k,\tilde{\delta}^k) = \project{\polar{\mcf{G}}}(\tilde{p}^k,-\tilde{p}^k,\tilde{r}^k,-\tilde{r}^k).
\]
Observing from \eqref{eqn:KF-KG:KG} and \eqref{KF-KG-polar:KG} that the sets $\mcf{G}$ and $\polar{\mcf{G}}$ are equivalent up to the permutation and sign of arguments, it follows from the properties of the projection point $(\tilde{p}^k,-\tilde{p}^k,\tilde{r}^k,-\tilde{r}^k)$ and \cite[\Prop~29.2(ii)]{Bauschke:2017:book} that
\[
  \tilde{x}^k = -\tilde{z}^k, \quad
  \tilde{\nu}^k = -\tilde{\lambda}^k, \quad
  \tilde{\tau}^k = -\tilde{\delta}^k, \quad
  \tilde{\kappa}^k = -\tilde{\eps}^k.
\]
Due to Moreau's decomposition, we also have
\[
  \tilde{x}^k + \tilde{\lambda}^k = \tilde{p}^k, \quad
  \tilde{\nu}^k + \tilde{z}^k = -\tilde{p}^k, \quad
  \tilde{\tau}^k + \tilde{\eps}^k = \tilde{r}^k, \quad
  \tilde{\kappa}^k + \tilde{\delta}^k = -\tilde{r}^k,
\]
and thus
\begin{equation}\label{eqn:nut_from_xt}
  \tilde{\nu}^k = -\tilde{p}^k + \tilde{x}^k, \quad
  \tilde{\kappa}^k = -\tilde{r}^k + \tilde{\tau}^k.
\end{equation}
It also follows from Moreau's decomposition that
\[
  \innerprod{\tilde{x}^k}{\tilde{p}^k-\tilde{x}^k} + \tilde{\tau}^k (\tilde{r}^k - \tilde{\tau}^k) = 0.
\]
Observing from the definitions of $\tilde{p}^k$ and $\tilde{r}^k$, and \eqref{eqn:nut_from_xt} that
\begin{align*}
  \tilde{x}^k - x^k &= \tilde{\nu}^k - \nu^k \\
  \tilde{\tau}^k - \tau^k &= \tilde{\kappa}^k - \kappa^k,
\end{align*}
we can write \eqref{eqn:dra-hsde:s-update} as
\begin{align*}
  s_x^{k+1} &= s_x^k + \tilde{x}^k - x^k = p^k + \tilde{x}^k - x^k \eqqcolon p^{k+1} \\
  s_\nu^{k+1} &= s_\nu^k + \tilde{\nu}^k - \nu^k = p^k + \tilde{x}^k - x^k = p^{k+1} \\
  s_\tau^{k+1} &= s_\tau^k + \tilde{\tau}^k - \tau^k = r^k + \tilde{\tau}^k - \tau^k \eqqcolon r^{k+1} \\
  s_\kappa^{k+1} &= s_\kappa^k + \tilde{\kappa}^k - \kappa^k = r^k + \tilde{\tau}^k - \tau^k = r^{k+1}.
\end{align*}
In other words, if the initial iterate of DRA satisfies \eqref{eqn:x-nu-init}, then it will be satisfied for all iterations $k\in\Nat$.
We can now simplify iteration~\eqref{eqn:dra-hsde} as
\begin{subequations}\label{eqn:dra-hsde-simpler}
\begin{align}
  (x^k,\nu^k,\tau^k,\kappa^k) &\gets \project{\mcf{F}} (p^k,p^k,r^k,r^k) \label{eqn:dra-hsde-simpler:PKf} \\
  (\tilde{p}^k,\tilde{r}^k) &\gets (2x^k-p^k, 2\tau^k-r^k) \\
  (\tilde{x}^k,\tilde{\nu}^k,\tilde{\tau}^k,\tilde{\kappa}) &\gets \project{\mcf{G}} \big( \tilde{p}^k, -\tilde{p}^k, \tilde{r}^k, -\tilde{r}^k \big) \label{eqn:dra-hsde-simpler:PKg} \\
  (p^{k+1},r^{k+1}) &\gets (p^k,r^k) + (\tilde{x}^k,\tilde{\tau}^k) - (x^k,\tau^k),
\end{align}
\end{subequations}
where $(\nu^k,\kappa^k,\tilde{\nu}^k,\tilde{\kappa}^k)$ satisfy \eqref{eqn:nu_from_x} and \eqref{eqn:nut_from_xt}.

Note that the simplification of the DRA iteration \eqref{eqn:dra-hsde} made by an appropriate algorithm initialization given in \eqref{eqn:x-nu-init} and using Moreau's decomposition to establish a relationship between inputs and outputs of the projection operators is similar to the approach taken in \cite[\S3.2]{O'Donoghue:2016}.

\subsection{Evaluating $\project{\mcf{F}}$}\label{subsec:dra:projF}

We now discuss how to evaluate the projection in \eqref{eqn:dra-hsde-simpler:PKf}.
It follows from Proposition~\ref{prop:project-Kf} that
\begin{equation}\label{eqn:PKF-update}
  (x^k,\nu^k,\tau^k,\kappa^k) \gets
  \begin{cases}
    ( \project{\closure{\dom}\widetilde{F}}(p^k,p^k,r^k), r^k ) & \tilde{F}(\project{\closure{\dom}\widetilde{F}}(p^k,p^k,r^k)) \le  -r^k \\
    ( \prox_{\mu^k\widetilde{F}}(p^k,p^k,r^k), r^k - \mu^k ) & \text{otherwise},
  \end{cases}
\end{equation}
where $\mu^k$ satisfies
\[
  \mu^k - r^k = \widetilde{F}( \prox_{\mu^k\widetilde{F}}(p^k,p^k,r^k) ).
\]
It follows from \eqref{eqn:nu_from_x} and \eqref{eqn:PKF-update} that $\kappa^k=r^k-\tau^k=r^k-\mu^k$, which means that $\mu^k=\tau^k$ and that the first case in \eqref{eqn:PKF-update} corresponds to $\tau^k=0$ and the second case to $\tau^k>0$.
We analyze these two cases separately.

\subsubsection{$\tau^k=0$}

Let us denote by
\[
  (\dot{x}^k,\dot{\nu}^k,\dot{\tau}^k) \eqdef \project{\closure{\dom}\widetilde{F}}(p^k,p^k,r^k).
\]
If $\tilde{F}(\dot{x}^k,\dot{\nu}^k,\dot{\tau}^k) \le -r^k$, then using \eqref{eqn:nu_from_x} and \eqref{eqn:PKF-update} it must be that
\[
  (x^k,\nu^k,\tau^k,\kappa^k) \gets (\dot{x}^k,\dot{\nu}^k,0,r^k),
\]
and thus $\dot{\tau}^k=0$ and
\[
  (\dot{x}^k,\dot{\nu}^k) = \project{\closure{\dom}(\recession{F})}(p^k,p^k).
\]
Hence, instead of considering the projection onto $\closure{\dom}\widetilde{F}$, we can focus on evaluating the projection onto $\closure{\dom}(\recession{F})$, \ie
\[
  \project{\closure{\dom}(\recession{F})}(p^k,p^k) = \big( \project{\closure{\dom}(\recession{f})}(p^k), \project{\closure{\dom}(\recession{f^*})}(p^k) \big).
\]

\paragraph{Projection onto domain of recession function.}

Setting
\[
  \hat{x}^k \gets \project{\closure{\dom}(\recession{f})}(p^k),
\]
we let $(\hat{x}^k,p^k-\hat{x}^k,0,r^k)$ be our candidate for $\project{\mcf{F}} (p^k,p^k,r^k,r^k)$.
To verify whether the candidate is indeed the projection of $(p^k,p^k,r^k,r^k)$ onto $\mcf{F}$, it must satisfy the following conditions:
\begin{align*}
  & (\hat{x}^k,p^k-\hat{x}^k,0,r^k) \in \mcf{F} \\
  & (p^k-\hat{x}^k,\hat{x}^k,r^k,0) \in \polar{\mcf{F}} \\
  & \innerprod{(\hat{x}^k,p^k-\hat{x}^k,0,r^k)}{(p^k-\hat{x}^k,\hat{x}^k,r^k,0)} = 0.
\end{align*}
Due to the structure of $\mcf{F}$ and $\polar{\mcf{F}}$, the first two conditions are equivalent and reduce to the following inequality:
\begin{equation}\label{eqn:recF-ineq}
  -r^k \ge (\recession{F})(\hat{x}^k,p^k-\hat{x}^k) = (\recession{f})(\hat{x}^k) + (\recession{f^*})(p^k-\hat{x}^k),
\end{equation}
while the third condition reduces to
\begin{equation}\label{eqn:cand-orth}
  \innerprod{\hat{x}^k}{p^k-\hat{x}^k} = 0,
\end{equation}
which is satisfied by construction due to Moreau's decomposition as $\hat{x}^k$ is the projection of $p^k$ onto a nonempty closed convex cone.

\paragraph{Projection onto recession cone of domain.}

Alternatively, we could set
\[
  \hat{\nu}^k \gets \project{\closure{\dom}(\recession{f^*})}(p^k),
\]
Using the following identity \cite[\App~A]{Banjac:2021}:
\[
  \closure{\dom}(\recession{f^*}) = \polar{(\recession{(\closure{\dom}f)})},
\]
along with the facts that the sets above are nonempty closed convex cones, and using Moreau's decomposition, we have
\[
  \hat{x}^k \gets p^k - \hat{\nu}^k
  = p^k - \project{\closure{\dom}(\recession{f^*})}(p^k)
  = p^k - \project{\polar{(\recession{(\closure{\dom}f)})}}(p^k)
  = \project{\recession{(\closure{\dom}f)}}(p^k).
\]
Similarly, we would need to confirm that the candidate $(\hat{x}^k,p^k-\hat{x}^k,0,r^k)$ is indeed the projection of $(p^k,p^k,r^k,r^k)$ onto $\mcf{F}$ by checking whether inequality \eqref{eqn:recF-ineq} is satisified.

\subsubsection{$\tau^k>0$}

If $(\recession{F})(\hat{x}^k,p^k-\hat{x}^k) > -r^k$, then it must be that $\tau^k>0$ and, due to \eqref{eqn:PKF-update}, we need to evaluate the proximity operator of $\mu^k\widetilde{F}$.
Using Proposition~\ref{prop:prox-ft} along with $\tau^k>0$, we need to solve the following system of equations:
\begin{align*}
  (x^k,\nu^k) &= \tau^k \prox_{(\mu^k/\tau^k)F}(p^k/\tau^k,p^k/\tau^k) \\
  \tau^k - r^k &= \mu^k F^* \big( \prox_{(\tau^k/\mu^k)F^*}(p^k/\mu^k,p^k/\mu^k) \big).
\end{align*}
Recalling that $\mu^k=\tau^k$ and eliminating $\nu^k$ via \eqref{eqn:nu_from_x}, we can reduce the above system to
\begin{align}\label{eqn:f-sys-eq}
\begin{split}
  x^k &= \tau^k \prox_f(p^k/\tau^k) \\
  \tau^k - r^k &= \tau^k f(x^k/\tau^k) + \tau^k f^*((p^k-x^k)/\tau^k) = \innerprod{x^k}{p^k-x^k} / \tau^k,
\end{split}
\end{align}
where the last equality follows from \eqref{x-tau-orth} (but also from \cite[\Lem~2.2(ii)]{Briceno-Arias:2024}).
As noted in \cite[Remark~3.1]{Briceno-Arias:2024c}, the unique solution to \eqref{eqn:f-sys-eq} is guaranteed to exist since it is equivalent to the following equality:
\[
  0 = \phi(\tau) \eqdef \tau - r^k - \widetilde{F} \big( \prox_{\tau\widetilde{F}}(p^k,p^k,r^k) \big),
\]
and due to \cite[\Lem~3.3]{Briceno-Arias:2024b}, $\widetilde{F} \big( \prox_{\tau\widetilde{F}}(p^k,p^k,r^k) \big)$ is continuous and decreasing in $\tau$ on $]0,+\infty[$, and thus $\phi$ is continuous, strictly increasing on $]0,+\infty[$, and
\begin{align*}
  \lim_{\tau\to 0} \phi(\tau) &= -r^k - \tilde{F} \big( \project{\closure{\dom}\widetilde{F}}(p^k,p^k,r^k) \big) < 0 \\
  \lim_{\tau\to \infty} \phi(\tau) &= +\infty.
\end{align*}

\subsection{Evaluating $\project{\mcf{G}}$}

Using similar arguments for the projection in \eqref{eqn:dra-hsde-simpler:PKg}, we let
\[
  \bar{x}^k \gets \project{\recession{(\closure{\dom}g)}}(\tilde{p}^k)
\]
and check the following inequality
\[
  -\tilde{r}^k \ge (\recession{G})(\bar{x}^k,\bar{x}^k-\tilde{p}^k) = (\recession{g})(\bar{x}^k) + (\recession{g^*})(\tilde{p}^k-\bar{x}^k).
\]
If the above inequality is not satisfied, then we proceed with solving the following system of equations:
\begin{align}\label{eqn:g-sys-eq}
\begin{split}
  \tilde{x}^k &= \tilde{\tau}^k \prox_g(\tilde{p}^k/\tilde{\tau}^k) \\
  \tilde{\tau}^k - \tilde{r}^k &= \tilde{\tau}^k g(\tilde{x}^k/\tilde{\tau}^k) + \tilde{\tau}^k g^*((\tilde{p}^k-\tilde{x}^k)/\tilde{\tau}^k) = \innerprod{\tilde{x}^k}{\tilde{p}^k-\tilde{x}^k} / \tilde{\tau}^k.
\end{split}
\end{align}

\begin{algorithm}[t]
\caption{DRA for HSDE of problem~\eqref{eqn:primal}.}
\label{alg:dra-hsde}
\begin{algorithmic}[1]
  \State {\bf given} initial values $p^1\in\mcf{H}$, $r^1\in\Re$
  \State $k \gets 1$
  \Repeat
    \State $\hat{x}^k \gets \project{\recession{(\closure{\dom}f)}}(p^k)$ or $\project{\closure{\dom}(\recession{f})}(p^k)$ \label{alg:dra-hsde:hat-x}
    \State $(x^k,\tau^k) \gets \begin{cases}
        (\hat{x}^k,0) & (\recession{f})(\hat{x}^k) + (\recession{f^*})(p^k-\hat{x}^k) \le -r^k \\
        \text{solve~\eqref{eqn:f-sys-eq}} & \text{otherwise}
      \end{cases}$ \label{alg:dra-hsde:x}
    \State $(\tilde{p}^k,\tilde{r}^k) \gets 2(x^k,\tau^k) - (p^k,r^k)$
    \State $\bar{x}^k \gets \project{\recession{(\closure{\dom}g)}}(\tilde{p}^k)$ or $\project{\closure{\dom}(\recession{g})}(\tilde{p}^k)$ \label{alg:dra-hsde:bar-x}
    \State $(\tilde{x}^k,\tilde{\tau}^k) \gets \begin{cases}
        (\bar{x}^k,0) & (\recession{g})(\bar{x}^k) + (\recession{g^*})(\tilde{p}^k-\bar{x}^k) \le -\tilde{r}^k \\
        \text{solve~\eqref{eqn:g-sys-eq}} & \text{otherwise}
      \end{cases}$ \label{alg:dra-hsde:tilde-x}
    \State $(p^{k+1},r^{k+1}) \gets (p^k,r^k) + (\tilde{x}^k,\tilde{\tau}^k) - (x^k,\tau^k)$
    \State $k \gets k+1$
  \Until{termination criterion is satisfied}
\end{algorithmic}
\end{algorithm}

\subsection{Final algorithm}

We summarize iteration \eqref{eqn:dra-hsde-simpler} in Algorithm~\ref{alg:dra-hsde} using the procedure for evaluating $\project{\mcf{F}}$ and $\project{\mcf{G}}$ outlined in the preceding sections.
Observe that in Steps \ref{alg:dra-hsde:hat-x} and \ref{alg:dra-hsde:bar-x} we have two ways of evaluating each of the candidates $\hat{x}^k$ and $\bar{x}^k$.
In practice, we would choose whichever variant is easier to evaluate.
Note, however, that the algorithm produces the same iterates $(x^k,\tau^k)$ regardless of how $\hat{x}^k$ is evaluated as $\project{\recession{(\closure{\dom}f)}}(p^k) \ne \project{\closure{\dom}(\recession{f})}(p^k)$ only if $\tau^k>0$, in which case the inequality in \eqref{eqn:recF-ineq} will not be satisfied and we will proceed with solving the system of equations in \eqref{eqn:f-sys-eq}.

\subsection{Convergence}

Since Algorithm~\ref{alg:dra-hsde} is DRA applied to problem~\eqref{eqn:hsde-composite}, whose solution set is a nonempty closed convex cone, its convergence follows from the general convergence result established in the seminal paper by Lions and Mercier \cite{Lions:1979}.
Since all iterates of Algorithm~\ref{alg:dra-hsde} satisfy \eqref{eqn:nu_from_x}, it also holds for the algorithm's fixed point $(p^\star,r^\star)$, \ie
\[
  p^\star = x^\star + \nu^\star, \quad
  r^\star = \tau^\star + \kappa^\star.
\]
We next show that the algorithm will not converge to zero (which is always a solution to the embedding) if the algorithm is initialized appropriately.

\subsubsection{Eliminating convergence to zero}\label{sec:elim-cvg-to-zero}

The following lemma, which is originally stated in a Euclidean space, but extends trivially to a general real Hilbert space, also applies to our iteration~\eqref{eqn:dra-hsde-simpler}.

\begin{lemma}[{\cite[\Lem~5.1]{O'Donoghue:2021}}]
  Let $w^1\in\mcf{H}$ and consider the sequence $w^{k+1}=T(w^k)$ for $k\in\Nat$ generated by $T:\mcf{H}\to\mcf{H}$.
  If
  \begin{enumerate}
    \item $T$ is positively homogeneous, \ie $T(tv)=tT(v)$ for any $t>0$, $v\in\mcf{H}$,
    \item $T$ has a nonzero fixed point $w^\star\in\mcf{H}$ which satisfies $\innerprod{w^\star}{w^1} > 0$,
    \item $T$ is nonexpansive toward any fixed point, \ie $\norm{T(v)-w^\star} \le \norm{v-w^\star}$ for any $v\in\mcf{H}$,
  \end{enumerate}
  then for all $k\in\Nat$,
  \[
    \norm{w^k} \ge \innerprod{w^\star}{w^1} / \norm{w^\star} > 0.
  \]
\end{lemma}

Let $T\in\mcf{H}\times\mcf{H}\times\Re^2\to\mcf{H}\times\mcf{H}\times\Re^2$ be the operator mapping $(p^k,p^k,r^k,r^k)$ to $(p^{k+1},p^{k+1},r^{k+1},r^{k+1})$ in iteration~\eqref{eqn:dra-hsde-simpler}.
It is easy to show that the operator is positively homogeneous since $\mcf{F}$ and $\mcf{G}$ are nonempty closed convex cones, and their projection operators are known to be positively homogeneous \cite[\Prop~29.29]{Bauschke:2017:book}.
Nonexpansiveness of the DRA operator is also well known \cite{Lions:1979}.
Hence, if the initial iterate is chosen such that $\innerprod{p^\star}{p^1} + r^\star r^1 > 0$, then $(p^k,r^k)$ will be bounded away from zero for all iterations $k\in\Nat$.

The authors in \cite{O'Donoghue:2016,O'Donoghue:2021} have shown that in their setup setting $p^1=0$ and $r^1=1$ would yield
\[
  \innerprod{p^\star}{p^1} + r^\star r^1 = r^\star = \tau^\star + \kappa^\star > 0
\]
as long as their exists a solution to the embedding at which either $\tau^\star$ or $\kappa^\star$ is nonzero.
This follows from the fact that in their setting both $\tau^\star$ and $\kappa^\star$ are guaranteed to be nonnegative.

We could use the same argument if $\widetilde{G}(x,\nu,\tau)$ is known to be lower-bounded by zero, \ie
\begin{equation}\label{eqn:Gtilde-nonneg}
  \forall (x,\nu,\tau) \quad  0 \le \widetilde{G}(x,\nu,\tau).
\end{equation}
Since any solution $(x^\star,\nu^\star,\tau^\star,\kappa^\star)$ to \eqref{eqn:hsde} lies in $\mcf{G}$, \ie
\[
  \widetilde{G}(x^\star,\nu^\star,\tau^\star) \le \kappa^\star,
\]
it follows that \eqref{eqn:Gtilde-nonneg} implies $\kappa^\star \ge 0$.
Hence, if there exists a solution to the embedding where either $\tau^\star\ge 0$ or $\kappa^\star\ge 0$ is nonzero, then the iterates $(p^k,r^k)$ will be bounded away from zero for all iterations $k\in\Nat$.

\subsection{Termination criteria}

Algorithm~\ref{alg:dra-hsde} can be seen as a projection method that seeks a point in the intersection of the sets $\mcf{F}$ and $\mcf{G}$ defined in \eqref{eqn:KF-KG}.
Hence, we can terminate the algorithm when the distance between $(x^k,\nu^k,\tau^k,\kappa^k)\in\mcf{F}$ and $(\tilde{x}^k,\tilde{\nu}^k,\tilde{\tau}^k,\tilde{\kappa}^k)\in\mcf{G}$ is small enough with respect to the magnitudes of these iterates.

Alternatively, we can define a termination criterion that is more specifically related to the violation of optimality and infeasibility conditions given in \eqref{eqn:zero-duality-gap} and \eqref{eqn:infeas-strict-ineq}, respectively.
For instance, let $(\tilde{x}^k,\tilde{\nu}^k,\tilde{\tau}^k,\tilde{\kappa}^k)\in\mcf{G}$ be our candidate solution to \eqref{eqn:hsde}.
Assuming $\tilde{\tau}^k>0$, we may want to evaluate the duality gap at this candidate solution, \ie
\[
  |f(\tilde{x}^k/\tilde{\tau}^k) + g(\tilde{x}^k/\tilde{\tau}^k) + f^*(\tilde{\nu}^k/\tilde{\tau}^k) + g^*(-\tilde{\nu}^k/\tilde{\tau}^k)|
\]
and terminate the algorithm when the duality gap is small enough.
Although it follows from \eqref{eqn:g-sys-eq} that
\begin{align*}
  (\tilde{x}^k/\tilde{\tau}^k) &\in \dom g \\
  (-\tilde{\nu}^k/\tilde{\tau}^k) &\in \dom g^*,
\end{align*}
we can still have
\begin{align*}
  (\tilde{x}^k/\tilde{\tau}^k) &\notin \dom f \\
  (\tilde{\nu}^k/\tilde{\tau}^k) &\notin \dom f^*,
\end{align*}
even for iterates that are arbitrarily close to these domains, which would result in the infinite duality gap.
A better way would be to evaluate the distance between $(\tilde{x}^k/\tilde{\tau}^k,\tilde{\nu}^k/\tilde{\tau}^k)$ and $\dom f \times \dom f^*$, and have some proxy for the duality gap evaluation, \eg we could evaluate $f$ and $f^\star$ at 
$x^k/\tau^k$ and $\nu^k/\tau^k$, respectively (assuming $\tau^k>0$), and require that the distance between $(\tilde{x}^k/\tilde{\tau}^k,\tilde{\nu}^k/\tilde{\tau}^k)$ and $(x^k/\tau^k,\nu^k/\tau^k)$ is small enough.

Similarly, we could let $(\tilde{x}^k,\tilde{\nu}^k)$ be our candidate for an infeasibility certificate and evaluate the violation of the inequalities in \eqref{eqn:infeas-strict-ineq}.
Again, we may want to first compute the distance between the candidate and the function domain, and then define a good proxy for the evaluation of the left-hand sides in \eqref{eqn:infeas-strict-ineq}.

We will show in \S\ref{sec:numerics} how we can define optimality and infeasibility residuals for a particular problem class introduced in \S\ref{sec:structured-cp}.

\section{Quadratic Cone Programming}\label{sec:qcp}

Consider the following quadratic cone program:
\begin{equation}\label{eqn:qcp-primal}
  \begin{array}{ll}
    \text{minimize} & \half x^T P x + c^T x \\
    \textrm{subject to} & Ax + z = b \\
    & z\in\mcf{K},
  \end{array}
  \tag{QCP}
\end{equation}
where $x\in\Re^n$ and $z\in\Re^m$ are optimization variables.
Problem data are given by matrices $P\in\symm_+^n$ and $A\in\Re^{m\times n}$, vectors $c\in\Re^n$ and $b\in\Re^m$, and a nonempty closed convex cone $\mcf{K}\subseteq\Re^m$.
Defining
\begin{align*}
  f(x,z) &= \half x^T P x + c^T x + \indicator{Ax+z=b}(x,z) \\
  g(x,z) &= \indicator{\mcf{K}}(z),
\end{align*}
we can represent problem~\eqref{eqn:qcp-primal} in form~\eqref{eqn:primal}.
The associated conjugate and recession functions are given in Lemma~\ref{lem:quad-affine-conj-rec} and Corollary~\ref{cor:cone-ind-conj-rec}.

\subsection{Evaluating algorithm steps: handling function $f$}\label{subsec:qcp:quad}

Evaluating Step~\ref{alg:dra-hsde:hat-x} of Algorithm~\ref{alg:dra-hsde} amounts to projecting $(p^k,s^k)$ onto $\recession{(\closure{\dom}f)}$, \ie
\[
  \begin{array}{ll}
    \text{minimize} & \half \norm{x-p^k}_2^2 + \half \norm{z-s^k}_2^2 \\
    \textrm{subject to} & Ax + z = 0,
  \end{array}
\]
which yields
\begin{equation}\label{eqn:qcp-linsys-proj}
  \begin{bmatrix} I & A^T \\ -A & I \end{bmatrix} \begin{bmatrix} \hat{x}^k \\ \hat{y}^k \end{bmatrix} = \begin{bmatrix} p^k \\ s^k \end{bmatrix},
\end{equation}
where $(\hat{w}^k,\hat{y}^k) = (p^k,s^k) - (\hat{x}^k,\hat{z}^k)$.
Evaluating Step~\ref{alg:dra-hsde:x} of Algorithm~\ref{alg:dra-hsde} requires first checking the following inequality:
\begin{equation}\label{eqn:qcp-ineq}
  -r_n \ge (\recession{f})(\hat{x}^k,\hat{z}^k) + (\recession{f^*})(p^k-\hat{x}^k,s^k-\hat{z}^k) = c^T \hat{x}^k + \indicator{\ker{P}}(\hat{x}^k) + b^T\hat{y}^k,
\end{equation}
where we removed indicator functions that equal to zero by construction due to \eqref{eqn:qcp-linsys-proj}.
If the inequality is satisfied, then $(x^k,y^k,\tau^k)=(\hat{x}^k,\hat{y}^k,0)$; otherwise, we need to solve \eqref{eqn:f-sys-eq}, which amounts to solving the following system of equations:
\begin{align}\label{eqn:qcp-Pf}
\begin{split}
  (M+I) \begin{bmatrix} x^k \\ y^k \end{bmatrix} &= \begin{bmatrix} p^k \\ s^k \end{bmatrix} - \tau^k \begin{bmatrix} c \\ b \end{bmatrix} \\
  \tau^k - r^k &= (p^k - x^k)^T x^k / \tau^k + (s^k-y^k)^T y^k / \tau^k,
\end{split}
\end{align}
where
\[
  M \eqdef \begin{bmatrix} P & A^T \\ -A & 0 \end{bmatrix}.
\]
Defining
\[
  v^k \eqdef (M+I)\inv \begin{bmatrix} p^k \\ s^k \end{bmatrix}
  \quad\text{and}\quad
  q \eqdef (M+I)\inv \begin{bmatrix} c \\ b \end{bmatrix},
\]
the solution to \eqref{eqn:qcp-Pf} satisfies the following equality:
\[
  \begin{bmatrix} x^k \\ y^k \end{bmatrix} = v^k - \tau^k q,
\]
where $\tau^k$ solves the following quadratic equation:
\[
  \underbrace{(1 + q^T q)}_{a} \tau^2 + \underbrace{(q^T \begin{bmatrix} p^k \\ s^k \end{bmatrix} - 2 q^T v^k - r^k)}_{b^k} \tau + \underbrace{(v^k - \begin{bmatrix} p^k \\ s^k \end{bmatrix})^T v^k}_{c^k} = 0.
\]
Since $a \eqdef 1 + q^T q > 0$ and
\[
  c^k \eqdef (v^k - \begin{bmatrix} p^k \\ s^k \end{bmatrix})^T v^k = -(v^k)^T M v^k \le 0,
\]
it follows that the discriminant of the quadratic equation is nonnegative and thus its roots are real.
Moreover, since $-4 a c^k \ge 0$, it follows that one root must be nonnegative and one must be nonpositive.
As we are interested in the nonnegative root, we have
\[
  \tau^k \gets \texttt{root}_+(p^k,s^k,r^k,v^k,q) \eqdef \big( -b^k + \sqrt{(b^k)^2 - 4ac^k} \big) / (2a).
\]

\subsubsection{Avoiding two matrix factorizations}

It seems from \eqref{eqn:qcp-linsys-proj} and \eqref{eqn:qcp-Pf} that computing $(x^k,y^k)$ requires solving two linear systems with different coefficient matrices.
However, if the nonnegative root of the quadratic equation yields $\tau^k=0$, then we must have
\[
  0 = c^k = - (v^k)^T M v^k = -(v_x^k)^T P v_x^k,
\]
which implies that $P v_x^k = 0$ and thus
\[
  v^k = (M+I)\inv \begin{bmatrix} p^k \\ s^k \end{bmatrix} = \begin{bmatrix} I & A^T \\ -A & I \end{bmatrix}\inv \begin{bmatrix} p^k \\ s^k \end{bmatrix} = \begin{bmatrix} \hat{x}^k \\ \hat{y}^k \end{bmatrix}.
\]
In other words, if $\tau^k=0$, then $(x^k,y^k)=v^k$ and there is no need to solve \eqref{eqn:qcp-linsys-proj} as we can use $v^k$ as a candidate for $(\hat{x}^k,\hat{y}^k)$.

\subsection{Evaluating algorithm steps: handling function $g$}

Evaluating Step~\ref{alg:dra-hsde:bar-x} of Algorithm~\ref{alg:dra-hsde} amounts to projecting $\tilde{s}^k$ onto $\mcf{K}$, \ie
\[
  (\bar{x}^k, \bar{z}^k) \gets \big( \tilde{p}^k, \project{\mcf{K}}(\tilde{s}^k) \big).
\]
Evaluating Step~\ref{alg:dra-hsde:tilde-x} of Algorithm~\ref{alg:dra-hsde} requires first checking the following inequality:
\[
  -\tilde{r}^k \ge (\recession{g})(\bar{x}^k,\bar{z}^k) + (\recession{g^*})(\tilde{p}^k-\bar{x}^k,\tilde{s}^k-\bar{z}^k) = \indicator{\mcf{K}}(\bar{z}^k) + \indicator{\{0\}}(-\bar{w}^k) + \indicator{\polar{\mcf{K}}}(-\bar{y}^k).
\]
Note that the right-hand side of the inequality equals zero by construction as $-\bar{y}^k = \tilde{s}^k - \bar{z}^k = \project{\polar{\mcf{K}}}(\tilde{s}^k)$ and $-\bar{w}^k = \tilde{p}^k - \bar{x}^k = 0$.
Hence, the inequality is satisfied if $\tilde{r}^k \le 0$, which yields
\[
  (\tilde{x}^k,\tilde{y}^k,\tilde{\tau}^k) \gets \big( \tilde{p}^k,-\project{\polar{\mcf{K}}}(\tilde{s}^k),0 \big).
\]
Otherwise, if $\tilde{r}^k > 0$, we need to solve \eqref{eqn:g-sys-eq}, which yields
\[
  (\tilde{x}^k,\tilde{y}^k,\tilde{\tau}^k) \gets \big( \tilde{p}^k,-\project{\polar{\mcf{K}}}(\tilde{s}^k),\tilde{r}^k \big).
\]
In other words, evaluating Step~\ref{alg:dra-hsde:tilde-x} of Algorithm~\ref{alg:dra-hsde} can be summarized as
\[
  (\tilde{x}^k,\tilde{y}^k,\tilde{\tau}^k) \gets \big( \tilde{p}^k,-\project{\polar{\mcf{K}}}(\tilde{s}^k),\max(\tilde{r}^k,0) \big).
\]
Since $(\tilde{p}^k,\tilde{s}^k,\tilde{r}^k) = 2(x^k,z^k,\tau^k) - (p^k,s^k,r^k)$, we can write
\begin{align*}
  \tilde{x}^k &\gets 2x^k - p^k \\
  \tilde{y}^k &\gets -\project{\polar{\mcf{K}}}(2z^k - s^k) = -\project{\polar{\mcf{K}}}\big(-(2y^k - s^k)\big) = \project{\dual{\mcf{K}}}(2y^k - s^k) \\
  \tilde{\tau}^k &\gets \max(2\tau^k-r^k,0),
\end{align*}
where we used $z^k = s^k - y^k$.

\begin{algorithm}[t]
\caption{SCS algorithm.}
\label{alg:dra-hsde-qcp}
\begin{algorithmic}[1]
  \State {\bf given} initial values $p^1\in\Re^n$, $s^1\in\Re^m$, $r^1\in\Re$
  \State {\bf compute} $q=(M+I)\inv (c,b)$
  \State $k \gets 1$
  \Repeat
    \State $v^k \gets (M+I)\inv (p^k,s^k)$
    \State $\tau^k \gets \texttt{root}_+(p^k, s^k, r^k, v^k, q)$
    \State $(x^k,y^k) \gets v^k - \tau^k q$
    \State $\tilde{x}^k \gets 2 x^k - p^k$
    \State $\tilde{y}^k \gets \project{\dual{\mcf{K}}} (2y^k-s^k)$
    \State $\tilde{\tau}^k \gets \max(2\tau^k-r^k,0)$
    \State $(p^{k+1},s^{k+1},r^{k+1}) \gets (p^k,s^k,r^k) + (\tilde{x}^k,\tilde{y}^k,\tilde{\tau}^k) - (x^k,y^k,\tau^k)$
    \State $k \gets k+1$
  \Until{termination criterion is satisfied}
\end{algorithmic}
\end{algorithm}

\subsection{SCS algorithm}
We summarize the DRA for HSDE of problem \eqref{eqn:qcp-primal} in Algorithm~\ref{alg:dra-hsde-qcp}, which is exactly the SCS algorithm proposed in \cite[\Alg~5.1]{O'Donoghue:2021}.
Note that we have arrived to the same algorithm taking a very different path.
The derivation of the algorithm conducted in \cite{O'Donoghue:2021} cannot easily generalize to non-conic constraints or non-quadratic objective functions as it is based on representing the problem as a linear complementarity problem.
In contrast, our derivation applies to optimization problems with non-conic constraints and non-quadratic objective functions, which we demonstrate in the following section.

\begin{remark}
  Since $g$ is the indicator function of a nonempty closed convex cone, it is easy to show that \eqref{eqn:Gtilde-nonneg} holds, which ensures that the iterates of Algorithm~\ref{alg:dra-hsde-qcp} are bounded away from zero if the initial iterate is set to $(p^1,s^1,r^1)=(0,0,1)$.
\end{remark}

\section{Structured Convex Optimization}\label{sec:structured-cp}

We are interested in solving the following convex optimization problem:
\begin{equation}\label{eqn:sqp}
\begin{array}{ll}
	\text{minimize} & \half x^T P x + c^T x + d^T |x| \\
	\mbox{subject to} & A x - b \in [l,u]
\end{array}
\end{equation}
where $x\in\Re^n$ is the optimization variable.
Problem data are given by matrices $P\in\symm_+^n$ and $A\in\Re^{m\times n}$, vectors $c\in\Re^n$, $d\in\Re_+^n$, and $b\in\Re^m$, and nonempty set $[l,u]\subseteq\Re^m$, where elements of $l$ can take value $-\infty$, and elements of $u$ can take value $+\infty$.

Defining
\begin{align*}
  f(x,z) &= \half x^T P x + c^T x + \indicator{-Ax+z=-b}(x,z) \\
  g(x,z) &= \support{[-d,d]}(x) + \indicator{[l,u]}(z),
\end{align*}
we can represent problem~\eqref{eqn:sqp} in form~\eqref{eqn:primal}.

\subsection{Optimality and infeasibility conditions}

Using Lemma~\ref{lem:quad-affine-conj-rec} and Lemma~\ref{lem:set-ind-conj-rec}, optimality conditions for \eqref{eqn:sqp} are given by
\begin{subequations}\label{eqn:sqp-opt-cond}
\begin{align}
  z &\in [l,u] \label{eqn:sqp-opt-cond:pfeas1} \\
  0 &= Ax - b - z \label{eqn:sqp-opt-cond:pfeas2} \\
  -w &\in [-d,d], \quad -y\in\polar{(\recession{[l,u]})} \label{eqn:sqp-opt-cond:dfeas1} \\
  0 &= P\lambda + c - w - A^T y \label{eqn:sqp-opt-cond:dfeas2} \\
  0 &= \half x^T P x + c^T x + d^T |x| + \half \lambda^T P \lambda - b^T y + \support{[l,u]}(-y), \label{eqn:sqp-opt-cond:dgap}
\end{align}
\end{subequations}
where \eqref{eqn:sqp-opt-cond:pfeas1}--\eqref{eqn:sqp-opt-cond:pfeas2} represent primal feasibility, \eqref{eqn:sqp-opt-cond:dfeas1}--\eqref{eqn:sqp-opt-cond:dfeas2} dual feasibility, and \eqref{eqn:sqp-opt-cond:dgap} zero duality gap.
Using \eqref{eqn:infeas-strict-ineq} along with Lemma~\ref{lem:quad-affine-conj-rec} and Lemma~\ref{lem:set-ind-conj-rec}, primal infeasibility conditions for \eqref{eqn:sqp} are given by
\begin{subequations}\label{eqn:qp-pinf-cond}
\begin{align}
  -\bar{y} &\in \polar{(\recession{[l,u]})} \\
  0 &> \support{[l,u]}(-\bar{y}) - b^T \bar{y} \\
  0 &= A^T \bar{y}
\end{align}
\end{subequations}
and dual infeasibility conditions are
\begin{subequations}\label{eqn:qp-dinf-cond}
\begin{align}
  \bar{z} &\in \recession{[l,u]} \\
  0 &> c^T \bar{x} + d^T |\bar{x}| \\
  0 &= P \bar{x} \\
  0 &= A \bar{x} - \bar{z}.
\end{align}
\end{subequations}

\subsection{Evaluating algorithm steps}

Note that evaluating Steps \ref{alg:dra-hsde:hat-x}--\ref{alg:dra-hsde:x} of Algorithm~\ref{alg:dra-hsde} can be done in the same way as shown in Section~\ref{subsec:qcp:quad} if we replace $(b,A)$ with $(-b,-A)$.
Computing $\bar{z}^k$ in Step~\ref{alg:dra-hsde:bar-x} amounts to projecting $\tilde{s}^k$ onto $\recession{[l,u]}$, \ie
\[
  \bar{z}^k_j = \project{\recession{[l_j,u_j]}} ( \tilde{s}^k_j )
  = \begin{cases}
    0 & l_j\in\Re, \, u_j\in\Re \\
    \max ( \tilde{s}^k_j, 0 ) & l_j\in\Re, \, u_j=+\infty \\
    \min ( \tilde{s}^k_j, 0 ) & l_j = -\infty, \, u_j\in\Re \\
    \tilde{s}^k_j & l_j = -\infty, \, u_j=+\infty.
  \end{cases}
\]
Evaluating Step~\ref{alg:dra-hsde:tilde-x} requires first checking the following inequality:
\[
  -\tilde{r}^k \ge d^T |\tilde{p}^k| + \support{[l,u]}(\tilde{s}^k-\bar{z}^k).
\]
If the inequality is satisfied, then $(\tilde{x}^k,\tilde{z}^k,\tilde{\tau}^k)=(\tilde{p}^k,\bar{z}^k,0)$; otherwise, we know that $\tilde{\tau}^k>0$ and we need to solve \eqref{eqn:g-sys-eq}, which yields the following system of equations:
\begin{align}\label{eqn:prox_box}
\begin{split}
  -\tilde{w}^k &= \tilde{\tau}^k \project{[-d,d]}(\tilde{p}^k/\tilde{\tau}^k) \\
  \tilde{z}^k &= \tilde{\tau}^k \project{[l,u]}(\tilde{s}^k/\tilde{\tau}^k) \\
  \tilde{\tau}^k - \tilde{r}^k &= -(\tilde{p}^k+\tilde{w}^k)^T \tilde{w}^k / \tilde{\tau}^k + (\tilde{s}^k-\tilde{z}^k)^T \tilde{z}^k / \tilde{\tau}^k,
\end{split}
\end{align}
where $\tilde{x}^k=\tilde{p}^k + \tilde{w}^k$.
Observe that, if $\tilde{s}^k_j/\tilde{\tau}^k\in[l_j,u_j]$, then $\tilde{z}^k_j=\tilde{s}^k_j$ and thus
\[
  (\tilde{s}^k_j - \tilde{z}^k_j \big) \tilde{z}^k_j / \tilde{\tau}^k = 0.
\]
On the other hand, if $\tilde{s}^k_j/\tilde{\tau}^k > u_j$, then $\tilde{z}^k_j=\tilde{\tau}^k u_j$ and thus
\[
  ( \tilde{s}^k_j - \tilde{z}^k_j ) \tilde{z}^k_j / \tilde{\tau}^k = ( \tilde{s}^k_j - \tilde{\tau}^k u_j ) u_j.
\]
Note that similar observations hold for $\tilde{w}^k$.
Defining sets of indices corresponding to lower- and upper-active bounds as
\begin{align*}
  \mcf{L}_w(\tilde{\tau}) &\eqdef \{ i \in \{ 1,\ldots,n\} \mid \tilde{p}^k_i/\tilde{\tau} < -d_i \} \\
  \mcf{U}_w(\tilde{\tau}) &\eqdef \{ i \in \{ 1,\ldots,n\} \mid \tilde{p}^k_i/\tilde{\tau} > d_i \} \\
  \mcf{L}_z(\tilde{\tau}) &\eqdef \{ j \in \{ 1,\ldots,m\} \mid \tilde{s}^k_j/\tilde{\tau} < l_j \} \\
  \mcf{U}_z(\tilde{\tau}) &\eqdef \{ j \in \{ 1,\ldots,m\} \mid \tilde{s}^k_j/\tilde{\tau} > u_j \}
\end{align*}
yields
\begin{align*}
  0 = \phi(\tilde{\tau}) = \tilde{\tau} - \tilde{r}^k
  & + \sum_{i\in\mcf{L}_w(\tilde{\tau})} d_i \big( \tilde{\tau} d_i + \tilde{p}^k_i \big)
  + \sum_{i\in\mcf{U}_w(\tilde{\tau})} d_i \big( \tilde{\tau} d_i - \tilde{p}^k_i \big) \\
  &+ \sum_{j\in\mcf{L}_z(\tilde{\tau})} l_j \big( \tilde{\tau} l_j - \tilde{s}^k_j \big)
  + \sum_{j\in\mcf{U}_z(\tilde{\tau})} u_j \big( \tilde{\tau} u_j - \tilde{s}^k_j \big),
\end{align*}
where $\phi:\Re_+\to\Re$ is a continuous strictly increasing piecewise affine function.
Assuming the active sets will not change, the zero of $\phi$ can be computed via
\[
  \tilde{\tau} = \frac{ \tilde{r}^k
  - \sum_{i\in\mcf{L}_w(\tilde{\tau})} \tilde{p}^k_i d_i
  + \sum_{i\in\mcf{U}_w(\tilde{\tau})} \tilde{p}^k_i d_i
  + \sum_{j\in\mcf{L}_z(\tilde{\tau})} \tilde{s}^k_j l_j
  + \sum_{j\in\mcf{U}_z(\tilde{\tau})} \tilde{s}^k_j u_j
  }{
    1
    + \sum_{i\in\mcf{L}_w(\tilde{\tau})} (d_i)^2
    + \sum_{i\in\mcf{U}_w(\tilde{\tau})} (d_i)^2
    + \sum_{j\in\mcf{L}_z(\tilde{\tau})} (l_j)^2
    + \sum_{j\in\mcf{U}_z(\tilde{\tau})} (u_j)^2
  }.
\]
If the expression above produces $\tilde{\tau} \le 0$ or a $\tilde{\tau}$ for which the active sets changed, then we need to iterate with a new candidate $\tilde{\tau}$.
Since the algorithm solves a sequence of similar problems with parameter $(\tilde{p},\tilde{s})$ that does not change too much between iterations $k$ and $k+1$ (especially when close to solution), we can use warm-starting to obtain a good initial guess for $\tilde{\tau}$.
Note that for every candidate $\tilde{\tau}$ we need to project $(\tilde{p}^k,\tilde{s}^k)/\tilde{\tau}$ onto $[-d,d]\times[l,u]$.

\begin{remark}
  The SCS solver handles box constraints on optimization variable $z$ by representing them as the intersection of the following sets \cite[\S6.2]{O'Donoghue:2021}:
  \[
    \mcf{K}_\text{box} \eqdef \lbrace (t,z) \in \Re \times \Re^m \mid tl \le z \le tu, \: t \ge 0 \rbrace
    \quad\text{and}\quad
    \lbrace (t,z) \in \Re \times \Re^m \mid t = 1 \rbrace.
  \]
  Projection onto $\mcf{K}_\text{box}$ can be done via Newton's method on the scalar variable $t$.
  Interestingly, this method yields a procedure equivalent to the one described above for solving the system of equations in \eqref{eqn:prox_box} when $d=0$.
\end{remark}

\begin{remark}
  If $0\notin[l,u]$, then $\support{[l,u]}(y)$ is not necessarily nonnegative, and thus we cannot guarantee that $\kappa^\star\ge 0$, which we used in \S\ref{sec:elim-cvg-to-zero} to ensure that the iterates of Algorithm~\ref{alg:dra-hsde} are bounded away from zero.
  However, for any $e\in[l,u]$ we can reformulate problem~\eqref{eqn:sqp} as
  \[
    \begin{array}{ll}
      \text{minimize} & \half x^T P x + c^T x + d^T |x| \\
      \mbox{subject to} & A x - (b+e) \in [l-e,u-e],
    \end{array}
  \]
  where $0\in[l-e,u-e]$.
\end{remark}

\section{Numerical Results}\label{sec:numerics}

Consider the following parametric problem:
\begin{equation}\label{eqn:parametric-qp}
  \begin{array}{ll}
    \text{minimize} & \half x_1^2 + x_1 - x_2 + \half |x_1| + \half |x_2| \\
    \textrm{subject to} & 0 \le x_1 + a x_2 \le u^z \\
    & 1 \le x_1 \le 3 \\
    & 1 \le x_2 \le u^x_2,
  \end{array}
\end{equation}
where $x_1\in\Re$ and $x_2\in\Re$ are optimization variables, and $a\in\Re$, $u^x_2 \ge 1$, and $u^z\ge 0$ are parameters.
A similar problem is considered in \cite[\S 6.1]{Banjac:2019} with the difference that we added the absolute value terms to the objective function.
The problem can be represented in form \eqref{eqn:sqp} with
\[
  P = \begin{bmatrix} 1 & 0 \\ 0 & 0 \end{bmatrix}, \:\:
  c = \begin{bmatrix} 1 \\ -1 \end{bmatrix}, \:\:
  d = \begin{bmatrix} 0.5 \\ 0.5 \end{bmatrix}, \:\:
  A = \begin{bmatrix}
    1 & a \\
    1 & 0 \\
    0 & 1 \\
  \end{bmatrix}, \:\:
  b = \begin{bmatrix} 0 \\ 1 \\ 1 \end{bmatrix}, \:\:
  l = \begin{bmatrix} 0 \\ 0 \\ 0 \end{bmatrix}, \:\:
  u = \begin{bmatrix} u^z \\ 2 \\ u^x_2-1 \end{bmatrix}.
\]
We will now discuss four scenarios that can occur depending on the values of parameters $a$, $u^x_2$ and $u^z$: (i) optimality, (ii) primal infeasibility, (iii) dual infeasibility, and (iv) simultaneous primal and dual infeasibility, and will show that Algorithm~\ref{alg:dra-hsde} correctly produces certificates for all four scenarios.
In all cases, we set the initial iterate $(p^1,s^1,r^1) = (0,0,1)$.

\subsection{Optimality}

\begin{figure}[t]
  \centering
  \begin{center}
    \begin{tikzpicture}
      \begin{axis}[
        width=\textwidth,
        height=0.5\textwidth,
        grid=both,
        grid style={line width=.1pt, draw=gray!30},
        major grid style={line width=.2pt,draw=gray!60},
        xlabel={Iteration $k$},
        ymin = 1e-10,
        ymax = 1e2,
        xmin = 0,
        xmax = 101,
        xtick distance = 20,
        ymode=log,
        legend entries={$r_\text{prim}^k$, $r_\text{dual}^k$, $r_\text{gap}^k$},
        legend cell align=left,
        legend style={at={(.975,.975)}, nodes={scale=1}, anchor=north east, fill=white, fill opacity=1, draw opacity=1,text opacity=1},
        every axis plot/.append style={ultra thick}
      ]
        \addplot [blue]       table[x=iter, y=pres, col sep=comma] {df_opt.csv};
        \addplot [red,dashed] table[x=iter, y=dres, col sep=comma] {df_opt.csv};
        \addplot [brown,dashdotted] table[x=iter, y=gap, col sep=comma] {df_opt.csv};
      \end{axis}
		\end{tikzpicture}
    \end{center}
	\caption{Convergence of optimality residuals for problem~\eqref{eqn:parametric-qp} with $a=1$, $u^x_2=3$, $u^z=5$.}
	\label{fig:qp-opt:res}
\end{figure}
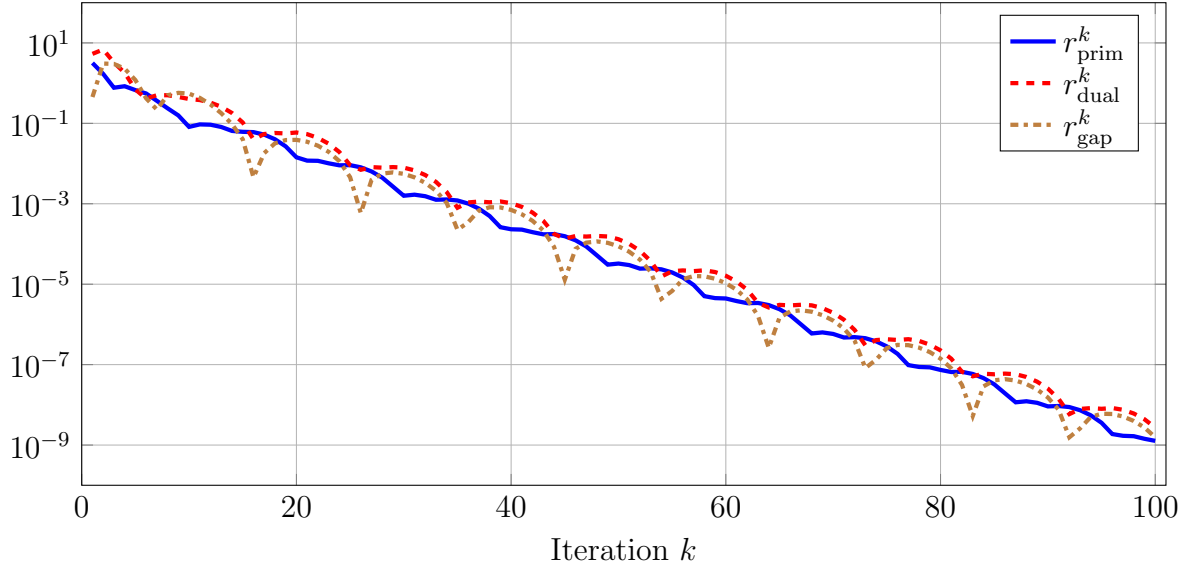

Consider problem~\eqref{eqn:parametric-qp} with parameters
\[
  a = 1, \quad
  u^x_2 = 3, \quad
  u^z = 5.
\]
When $\tilde{\tau}^k>0$, we define our candidate solution as:
\begin{align*}
  x_\text{opt}^k &= \tilde{x}^k/\tilde{\tau}^k \\
  z_\text{opt}^k &= \tilde{z}^k/\tilde{\tau}^k \\
  w_\text{opt}^k &= (\tilde{x}^k-\tilde{p}^k)/\tilde{\tau}^k \\
  y_\text{opt}^k &= (\tilde{z}^k-\tilde{s}^k)/\tilde{\tau}^k.
\end{align*}
We know that
\begin{align*}
  z_\text{opt}^k &= \project{[l,u]}(\tilde{s}^k/\tilde{\tau}^k) \in [l,u] \\
  -w_\text{opt}^k &= \project{[-d,d]}(\tilde{p}^k/\tilde{\tau}^k) \in [-d,d] \\
  -y_\text{opt}^k &= \prox_{\support{[l,u]}}(\tilde{s}^k/\tilde{\tau}^k) \in \dom \support{[l,u]} \subseteq \polar{(\recession{[l,u]})}
\end{align*}
and based on \eqref{eqn:sqp-opt-cond} we define the following optimality residuals:
\begin{align*}
  r_\text{prim}^k &\eqdef \norm{A x_\text{opt}^k - b - z_\text{opt}^k}_\infty \\
  r_\text{dual}^k &\eqdef \norm{P x_\text{opt}^k + c - w_\text{opt}^k - A^T y_\text{opt}^k}_\infty \\
  r_\text{gap}^k &\eqdef |(x_\text{opt}^k)^T P x_\text{opt}^k + c^T x_\text{opt}^k + d^T |x_\text{opt}^k| - b^T y_\text{opt}^k + \support{[l,u]}(-y_\text{opt}^k)|.
\end{align*}
Algorithm~\ref{alg:dra-hsde} produces a sequence $\seq{x_\text{opt}^k,z_\text{opt}^k,w_\text{opt}^k,y_\text{opt}^k}$ that converges to
\begin{align*}
  x_\text{opt}^\star &= (1,3) \\
  z_\text{opt}^\star &= (4,1,3) \\
  w_\text{opt}^\star &= (-0.5,-0.5) \\
  y_\text{opt}^\star &= (0, 2.5, -0.5).
\end{align*}
We show convergence of optimality residuals in Figure~\ref{fig:qp-opt:res}.

\subsection{Primal infeasibility}\label{sec:numerics:pinf}

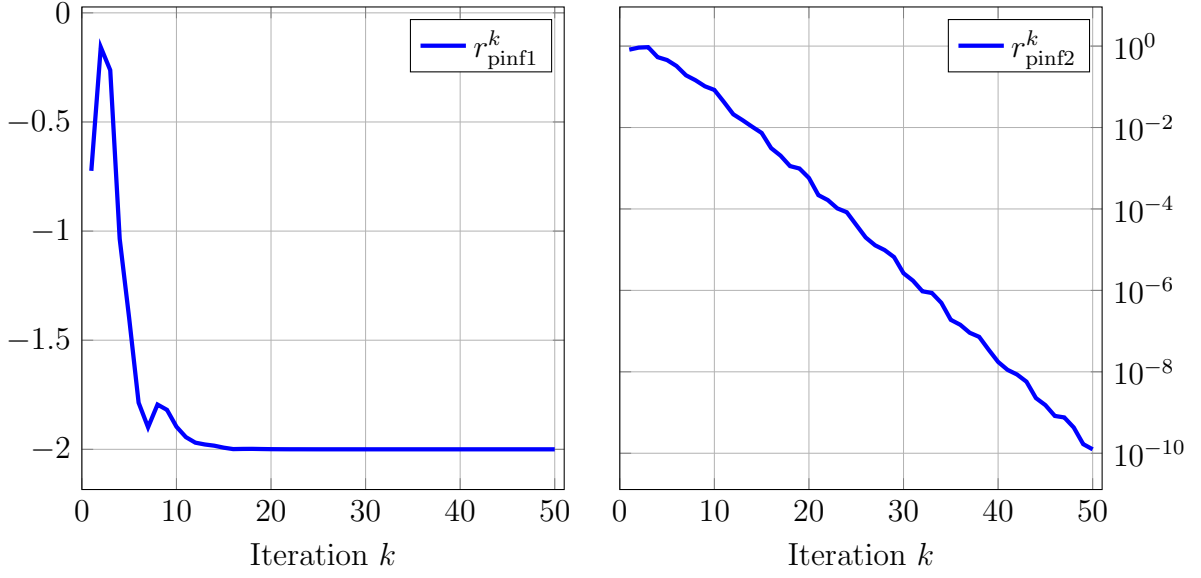
\begin{figure}[t]
  \centering
  \begin{center}
    \begin{tikzpicture}
      \begin{axis}[
        width=0.5\textwidth,
        height=0.5\textwidth,
        grid=both,
        grid style={line width=.1pt, draw=gray!30},
        major grid style={line width=.2pt,draw=gray!60},
        xlabel={Iteration $k$},
        xmin = 0,
        xmax = 51,
        xtick distance = 10,
        legend entries={$r_\text{pinf1}^k$},
        legend cell align=left,
        legend style={at={(.975,.975)}, nodes={scale=1}, anchor=north east, fill=white, fill opacity=1, draw opacity=1,text opacity=1},
        every axis plot/.append style={ultra thick}
		  ]
        \addplot [blue] table[x=iter, y=pinf1, col sep=comma] {df_pinf.csv};
      \end{axis}
    \end{tikzpicture}
    \hspace{0em}
    \begin{tikzpicture}
      \begin{axis}[
        width=0.5\textwidth,
        height=0.5\textwidth,
        grid=both,
        grid style={line width=.1pt, draw=gray!30},
        major grid style={line width=.2pt,draw=gray!60},
        xlabel={Iteration $k$},
        xmin = 0,
        xmax = 51,
        xtick distance = 10,
        yticklabel pos=right,
        ymode=log,
        legend entries={$r_\text{pinf2}^k$},
        legend cell align=left,
        legend style={at={(.975,.975)}, nodes={scale=1}, anchor=north east, fill=white, fill opacity=1, draw opacity=1,text opacity=1},
        every axis plot/.append style={ultra thick}
		  ]
        \addplot [blue] table[x=iter, y=pinf2, col sep=comma] {df_pinf.csv};
      \end{axis}
    \end{tikzpicture}
  \end{center}
  \caption{Convergence of primal infeasibility residuals for problem~\eqref{eqn:parametric-qp} with $a=1$, $u^x_2=3$, $u^z=0$.}
  \label{fig:qp-pinf:res}
\end{figure}

Consider problem~\eqref{eqn:parametric-qp} with parameters
\[
  a = 1, \quad
  u^x_2 = 3, \quad
  u^z = 0.
\]
We define
\[
  \delta^k \eqdef \norm{\project{\polar{(\recession{[l,u]})}}(\tilde{s}^k-\tilde{z}^k)}_\infty.
\]
When $\delta^k > 0$, we define our candidate primal infeasibility certificate as:
\[
  y_\text{pinf}^k = -\project{\polar{(\recession{[l,u]})}}(\tilde{s}^k-\tilde{z}^k) / \delta^k \in -\polar{(\recession{[l,u]})}
\]
and based on \eqref{eqn:qp-pinf-cond} we define the following primal infeasibility residuals:
\begin{align*}
  r_\text{pinf1}^k &\eqdef -b^T y_\text{pinf}^k + \support{[l,u]}(-y_\text{pinf}^k) \\
  r_\text{pinf2}^k &\eqdef \norm{A^T y_\text{pinf}^k}_\infty.
\end{align*}
Algorithm~\ref{alg:dra-hsde} produces a sequence $\seq{y_\text{pinf}^k}$ that converges to
\[
  y_\text{pinf}^\star = (-1,1,1).
\]
We show convergence of primal infeasibility residuals in Figure~\ref{fig:qp-pinf:res}.

\subsection{Dual infeasibility}\label{sec:numerics:dinf}

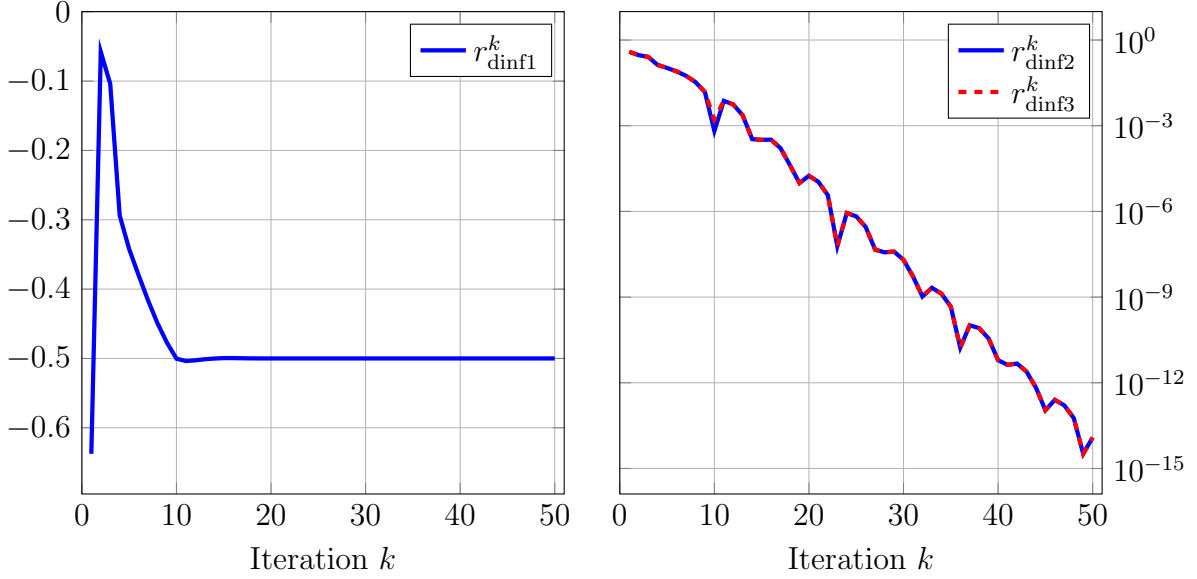
\begin{figure}[t]
  \centering
  \begin{center}
    \begin{tikzpicture}
      \begin{axis}[
        width=0.5\textwidth,
        height=0.5\textwidth,
        grid=both,
        grid style={line width=.1pt, draw=gray!30},
        major grid style={line width=.2pt,draw=gray!60},
        xlabel={Iteration $k$},
        xmin = 0,
        xmax = 51,
        legend entries={$r_\text{dinf1}^k$},
        legend cell align=left,
        legend style={at={(.975,.975)}, nodes={scale=1}, anchor=north east, fill=white, fill opacity=1, draw opacity=1,text opacity=1},
        every axis plot/.append style={ultra thick}
		  ]
        \addplot [blue] table[x=iter, y=dinf1, col sep=comma] {df_dinf.csv};
      \end{axis}
    \end{tikzpicture}
    \hspace{0em}
    \begin{tikzpicture}
      \begin{axis}[
        width=0.5\textwidth,
        height=0.5\textwidth,
        grid=both,
        grid style={line width=.1pt, draw=gray!30},
        major grid style={line width=.2pt,draw=gray!60},
        xlabel={Iteration $k$},
        xmin = 0,
        xmax = 51,
        yticklabel pos=right,
        ymode=log,
        legend entries={$r_\text{dinf2}^k$, $r_\text{dinf3}^k$},
        legend cell align=left,
        legend style={at={(.975,.975)}, nodes={scale=1}, anchor=north east, fill=white, fill opacity=1, draw opacity=1,text opacity=1},
        every axis plot/.append style={ultra thick}
		  ]
        \addplot [blue] table[x=iter, y=dinf2, col sep=comma] {df_dinf.csv};
        \addplot [red, dashed] table[x=iter, y=dinf3, col sep=comma] {df_dinf.csv};
      \end{axis}
    \end{tikzpicture}
  \end{center}
  \caption{Convergence of dual infeasibility residuals for problem~\eqref{eqn:parametric-qp} with $a=0$, $u^x_2=+\infty$, $u^z=2$.}
  \label{fig:qp-dinf:res}
\end{figure}

Consider problem~\eqref{eqn:parametric-qp} with parameters
\[
  a = 0, \quad
  u^x_2 = +\infty, \quad
  u^z = 2.
\]
We define
\[
  \gamma^k \eqdef \max(
    \norm{\tilde{x}^k}_\infty,
    \norm{\project{\recession{[l,u]}}(\tilde{z}^k)}_\infty
  ).
\]
When $\gamma^k > 0$, we define our candidate dual infeasibility certificate as:
\begin{align*}
  x_\text{dinf}^k &= \tilde{x}^k / \gamma^k \\
  s_\text{dinf}^k &= \project{\recession{[l,u]}}(\tilde{z}^k) / \gamma^k \in \recession{[l,u]}
\end{align*}
and based on \eqref{eqn:qp-dinf-cond} we define the following dual infeasibility residuals:
\begin{align*}
  r_\text{dinf1}^k &\eqdef c^T x_\text{dinf}^k + d^T |x_\text{dinf}^k| \\
  r_\text{dinf2}^k &\eqdef \norm{P x_\text{dinf}^k}_\infty \\
  r_\text{dinf3}^k &\eqdef \norm{A x_\text{dinf}^k - z_\text{dinf}^k}_\infty.
\end{align*}
Algorithm~\ref{alg:dra-hsde} produces a sequence $\seq{x_\text{dinf}^k,z_\text{dinf}^k}$ that converges to
\begin{align*}
  x_\text{dinf}^\star &= (0, 1) \\
  z_\text{dinf}^\star &= (0, 0, 1).
\end{align*}
We show convergence of dual infeasibility residuals in Figure~\ref{fig:qp-dinf:res}.

\subsection{Simultaneous primal and dual infeasibility}

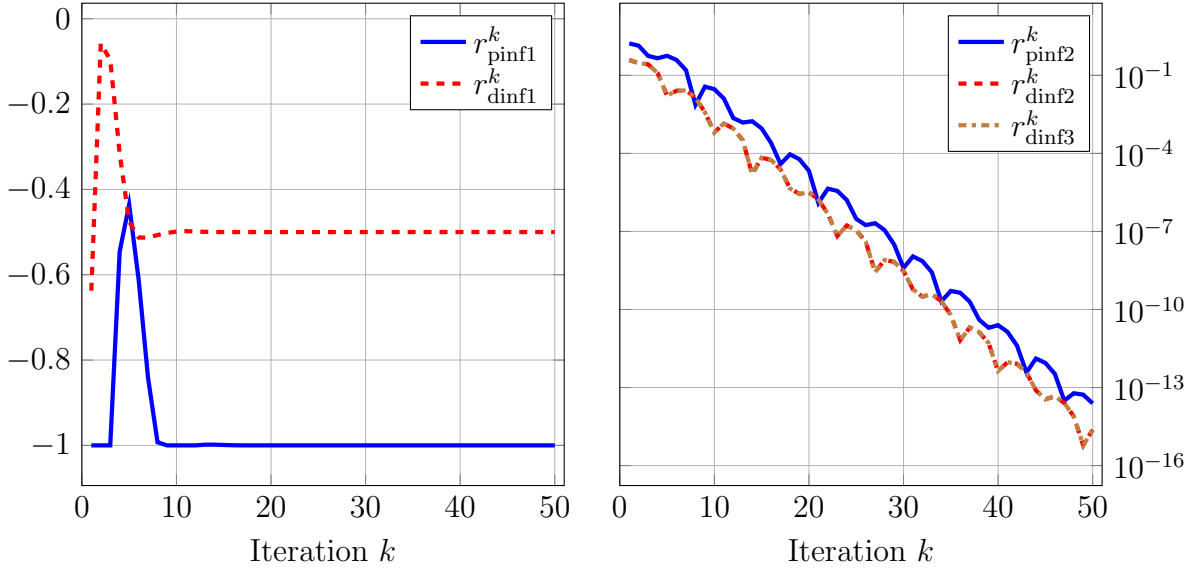
\begin{figure}[t]
  \centering
  \begin{center}
    \begin{tikzpicture}
      \begin{axis}[
        width=0.5\textwidth,
        height=0.5\textwidth,
        grid=both,
        grid style={line width=.1pt, draw=gray!30},
        major grid style={line width=.2pt,draw=gray!60},
        xlabel={Iteration $k$},
        xmin = 0,
        xmax = 51,
        legend entries={$r_\text{pinf1}^k$, $r_\text{dinf1}^k$},
        legend cell align=left,
        legend style={at={(.975,.975)}, nodes={scale=1}, anchor=north east, fill=white, fill opacity=1, draw opacity=1,text opacity=1},
        every axis plot/.append style={ultra thick}
		  ]
        \addplot [blue] table[x=iter, y=pinf1, col sep=comma] {df_pdinf.csv};
        \addplot [red, dashed] table[x=iter, y=dinf1, col sep=comma] {df_pdinf.csv};
      \end{axis}
    \end{tikzpicture}
    \hspace{0em}
    \begin{tikzpicture}
      \begin{axis}[
        width=0.5\textwidth,
        height=0.5\textwidth,
        grid=both,
        grid style={line width=.1pt, draw=gray!30},
        major grid style={line width=.2pt,draw=gray!60},
        xlabel={Iteration $k$},
        xmin = 0,
        xmax = 51,
        yticklabel pos=right,
        ymode=log,
        legend entries={$r_\text{pinf2}^k$, $r_\text{dinf2}^k$, $r_\text{dinf3}^k$},
        legend cell align=left,
        legend style={at={(.975,.975)}, nodes={scale=1}, anchor=north east, fill=white, fill opacity=1, draw opacity=1,text opacity=1},
        every axis plot/.append style={ultra thick}
		  ]
        \addplot [blue] table[x=iter, y=pinf2, col sep=comma] {df_pdinf.csv};
        \addplot [red, dashed] table[x=iter, y=dinf2, col sep=comma] {df_pdinf.csv};
        \addplot [brown, dashdotted] table[x=iter, y=dinf3, col sep=comma] {df_pdinf.csv};
      \end{axis}
    \end{tikzpicture}
  \end{center}
  \caption{Convergence of primal and dual infeasibility residuals for problem~\eqref{eqn:parametric-qp} with $a=0$, $u^x_2=+\infty$, $u^z=0$.}
  \label{fig:qp-pdinf:res}
\end{figure}

Consider problem~\eqref{eqn:parametric-qp} with parameters
\[
  a = 0, \quad
  u^x_2 = +\infty, \quad
  u^z = 0.
\]
We define candidates for primal and dual infeasibility, as well as the primal and dual infeasibility residuals in the same way as in \S\ref{sec:numerics:pinf} and \S\ref{sec:numerics:dinf}.
Algorithm~\ref{alg:dra-hsde} produces a sequence $\seq{y_\text{pinf}^k,x_\text{dinf}^k,z_\text{dinf}^k}$ that converges to
\begin{align*}
  y_\text{pinf}^\star &= (-1,1,0) \\
  x_\text{dinf}^\star &= (0,1) \\
  z_\text{dinf}^\star &= (0,0,1).
\end{align*}
We show convergence of infeasibility residuals in Figure~\ref{fig:qp-pdinf:res}.

\appendices

\newpage
\section{Supporting Results}

\begin{lemma}\label{lem:FG_conjugate}
Let $f$ and $g$ be functions in $\Gamma_0(\mcf{H})$, and let $F\colon\mcf{H}\times\mcf{H} \to \left]-\infty,+\infty\right]$ and $G\colon\mcf{H}\times\mcf{H} \to \left]-\infty,+\infty\right]$ be given by
\begin{align*}
  F(x,\nu) &\coloneqq f(x) + f^*(\nu) \\
  G(x,\nu) &\coloneqq g(x) + g^*(-\nu).
\end{align*}
Then their Fenchel conjugates $F^*\colon\mcf{H}\times\mcf{H} \to \left]-\infty,+\infty\right]$ and $G^*\colon\mcf{H}\times\mcf{H} \to \left]-\infty,+\infty\right]$ are given by
\begin{align*}
  F^*(\lambda,z) &= F(z,\lambda) \\
  G^*(\lambda,z) &= G(-z,-\lambda).
\end{align*}
\end{lemma}
\begin{proof}
Using the definition of Fenchel conjugate function, we can write
\begin{align*}
  F^*(\lambda,z) &= \sup_{(x,\nu)} \left( \innerprod{(x,\nu)}{(\lambda,z)} - f(x) - f^*(\nu) \right) \\
  &= \sup_x \left( \innerprod{x}{\lambda} - f(x) \right)
  + \sup_\nu \left( \innerprod{\nu}{z} - f^*(\nu) \right) \\
  &= f^*(\lambda) + f(z) \\
  &= F(z,\lambda).
\end{align*}
Similarly,
\begin{align*}
  G^*(\lambda,z) &= \sup_{(x,\nu)} \left( \innerprod{(x,\nu)}{(\lambda,z)} - g(x) - g^*(-\nu) \right) \\
  &= \sup_x \left( \innerprod{x}{\lambda} - g(x) \right)
  + \sup_\nu \left( \innerprod{-\nu}{-z} - g^*(-\nu) \right) \\
  &= g^*(\lambda) + g(-z) \\
  &= G(-z,-\lambda).
  \qedhere
\end{align*}
\end{proof}

\begin{lemma}\label{lem:quad-affine-conj-rec}
Let $f\in\Gamma_0(\mcf{H}_1\times\mcf{H}_2)$ be given by
\[
  f(x,z) = \half \innerprod{x}{Px} + \innerprod{c}{x} + \indicator{Ax+z=b}(x,z),
\]
where $P\colon\mcf{H}_1\to\mcf{H}_1$ is a monotone self-adjoint bounded linear operator, $c\in\mcf{H}_1$, $A\colon\mcf{H}_1\to\mcf{H}_2$ a bounded linear operator, and $b\in\mcf{H}_2$.
We assume that $\range{P}$ and $\range{A}$ are closed.
Then
\begin{align*}
  f^*(w,y) &= \half \innerprod{w-A^*y-c}{P^\dagger (w-A^*y-c)} + \indicator{\range{P}}(w-A^*y-c) + \innerprod{b}{y} \\
  (\recession{f})(\bar{x},\bar{z}) &= \innerprod{c}{\bar{x}} + \indicator{\ker{P}}(\bar{x}) + \indicator{Ax+z=0}(\bar{x},\bar{z}) \\
  (\recession{f^*})(\bar{w},\bar{y}) &= \innerprod{b}{\bar{y}} + \indicator{\{0\}}(\bar{w}-A^*\bar{y}),
\end{align*}
where $P^\dagger$ is the Moore-Penrose inverse of $P$.
Moreover,
\begin{align*}
  \closure{\dom} f &= \lbrace (x,z)\in\mcf{H}_1\times\mcf{H}_2 \mid Ax + z = b \rbrace \\
  \recession{(\closure{\dom}f)} &= \lbrace (\bar{x},\bar{z})\in\mcf{H}_1\times\mcf{H}_2 \mid A\bar{x} + \bar{z} = 0 \rbrace \\
  \closure{\dom}(\recession{f}) &= \lbrace (\bar{x},\bar{z})\in\mcf{H}_1\times\mcf{H}_2 \mid P\bar{x} = 0, \, A\bar{x} + \bar{z} = 0 \rbrace \\
  \closure{\dom} f^* &= \lbrace (w,y)\in\mcf{H}_1\times\mcf{H}_2 \mid w-A^*y-c \in \range P \rbrace \\
  \recession{(\closure{\dom}f^*)} &= \lbrace (\bar{w},\bar{y})\in\mcf{H}_1\times\mcf{H}_2 \mid \bar{w}-A^*\bar{y} \in \range P \rbrace \\
  \closure{\dom}(\recession{f^*}) &= \lbrace (\bar{w},\bar{y})\in\mcf{H}_1\times\mcf{H}_2 \mid \bar{w}-A^*\bar{y} = 0 \rbrace.
  \qedhere
\end{align*}
\end{lemma}
\begin{proof}
The expression for $f^*$ is a trivial extension of a result shown in \cite[\App~A]{Banjac:2021} for $b=0$.
The expressions for recession functions follow from the expressions for $\dom f$ and $\dom f^*$, and identities $\recession{f}=\support{\dom f^*}$ and $\recession{f^*}=\support{\dom f}$ \cite[\Prop~13.49]{Bauschke:2017:book}.
\end{proof}

\begin{lemma}\label{lem:set-ind-conj-rec}
Let $g\in\Gamma_0(\mcf{H}_1\times\mcf{H}_2)$ be given by
\[
  g(x,z) = \support{\setB}(x) + \indicator{\setC}(z),
\]
where $\setB\subseteq\mcf{H}_1$ and $\setC\subseteq\mcf{H}_2$ are nonempty closed convex sets.
Then
\begin{align*}
  g^*(w,y) &= \indicator{\setB}(w) + \support{\setC}(y) \\
  (\recession{g})(\bar{x},\bar{z}) &= \support{\setB}(\bar{x}) + \indicator{\recession{\setC}}(\bar{z}) \\
  (\recession{g^*})(\bar{w},\bar{y}) &= \indicator{\recession{\setB}}(\bar{w}) + \support{\setC}(\bar{y}).
\end{align*}
Moreover,
\begin{align*}
  \closure{\dom}              g &= \polar{(\recession{\setB})} \times \setC \\
  \recession{(\closure{\dom}g)} &= \polar{(\recession{\setB})} \times \recession{\setC} \\
  \closure{\dom}(\recession{g}) &= \polar{(\recession{\setB})} \times \recession{\setC} \\
  \closure{\dom} g^*              &= \setB             \times \polar{(\recession{\setC})} \\
  \recession{(\closure{\dom}g^*)} &= \recession{\setB} \times \polar{(\recession{\setC})} \\
  \closure{\dom}(\recession{g^*}) &= \recession{\setB} \times \polar{(\recession{\setC})}.
\end{align*}
\end{lemma}
\begin{proof}
Results follow directly from the following well-known facts:
\begin{align*}
  \indicator{\setD}^* &= \support{\setD} \\
  \recession{\indicator{\setD}} &= \indicator{\recession{\setD}} \\
  \recession{h} &= \support{\dom h^*}.
\end{align*}
where $\setD\subseteq\mcf{H}$ is a nonempty closed convex set and $h\in\Gamma_0(\mcf{H})$.
\end{proof}

\begin{corollary}\label{cor:cone-ind-conj-rec}
Let $g\in\Gamma_0(\mcf{H}_1\times\mcf{H}_2)$ be given by
\[
  g(x,z) = \indicator{\mcf{K}}(z),
\]
where $\mcf{K}\subseteq\mcf{H}_2$ is a nonempty closed convex cone.
Then
\begin{align*}
  g^*(w,y) &= \indicator{\{0\}}(w) + \indicator{\polar{\mcf{K}}}(y) \\
  (\recession{g})(\bar{x},\bar{z}) &= \indicator{\mcf{K}}(\bar{z}) \\
  (\recession{g^*})(\bar{w},\bar{y}) &= \indicator{\{0\}}(\bar{w}) + \indicator{\polar{\mcf{K}}}(\bar{y}).
\end{align*}
Moreover,
\begin{align*}
  \closure{\dom} g &= \mcf{H}_1 \times \mcf{K} \\
  \recession{(\closure{\dom}g)} &= \mcf{H}_1 \times \mcf{K} \\
  \closure{\dom}(\recession{g}) &= \mcf{H}_1 \times \mcf{K} \\
  \closure{\dom} g^* &= \{0\} \times \polar{\mcf{K}} \\
  \recession{(\closure{\dom}g^*)} &= \{0\} \times \polar{\mcf{K}} \\
  \closure{\dom}(\recession{g^*}) &= \{0\} \times \polar{\mcf{K}}.
\end{align*}
\end{corollary}

\bibliographystyle{alpha}
\bibliography{refs}

\newcommand{\etalchar}[1]{$^{#1}$}
\begin{thebibliography}{BAVV24b}

\bibitem[ADH{\etalchar{+}}24]{Applegate:2026}
D.~Applegate, M.~D\'{\i}az, O.~Hinder, H.~Lu, M.~Lubin, B.~O'Donoghue, and W.~Schudy.
\newblock {PDLP}: a practical first-order method for large-scale linear programming.
\newblock {\em Mathematical Programming Computation}, 2024.

\bibitem[ADLL24]{Applegate:2024}
D.~Applegate, M.~D\'{\i}az, H.~Lu, and M.~Lubin.
\newblock Infeasibility detection with primal-dual hybrid gradient for large-scale linear programming.
\newblock {\em SIAM Journal on Optimization}, 34(1):459--484, 2024.

\bibitem[ADV26]{cvxopt}
M.~S. Andersen, J.~Dahl, and L.~Vandenberghe.
\newblock {\em {CVXOPT}: a {P}ython package for convex optimization}, 2026.

\bibitem[AY99]{Andersen:1999}
E.~D. Andersen and Y.~Ye.
\newblock On a homogeneous algorithm for the monotone complementarity problem.
\newblock {\em Mathematical Programming}, 84(2):375--399, 1999.

\bibitem[BACS24]{Briceno-Arias:2024b}
L.~M. Brice{\~n}o-Arias, P.~L. Combettes, and F.~J. Silva.
\newblock Proximity operators of perspective functions with nonlinear scaling.
\newblock {\em SIAM Journal on Optimization}, 34(4):3212--3234, 2024.

\bibitem[Ban21]{Banjac:2021}
G.~Banjac.
\newblock On the minimal displacement vector of the {D}ouglas-{R}achford operator.
\newblock {\em Operations Research Letters}, 49(2):197--200, 2021.

\bibitem[BAVV24a]{Briceno-Arias:2024}
L.~M. Brice{\~n}o-Arias and C.~Vivar-Vargas.
\newblock Enhanced computation of the proximity operator for perspective functions.
\newblock {\em Journal of Optimization Theory and Applications}, 200(3):1078--1099, 2024.

\bibitem[BAVV24b]{Briceno-Arias:2024c}
L.~M. Brice{\~n}o-Arias and C.~Vivar-Vargas.
\newblock Projection onto cones generated by epigraphs of perspective functions.
\newblock {\em arXiv:2411.08000}, 2024.

\bibitem[BC17]{Bauschke:2017:book}
H.~H. Bauschke and P.~L. Combettes.
\newblock {\em Convex Analysis and Monotone Operator Theory in Hilbert Spaces}.
\newblock Springer International Publishing, 2nd edition, 2017.

\bibitem[BGSB19]{Banjac:2019}
G.~Banjac, P.~Goulart, B.~Stellato, and S.~Boyd.
\newblock Infeasibility detection in the alternating direction method of multipliers for convex optimization.
\newblock {\em Journal of Optimization Theory and Applications}, 183(2):490--519, 2019.

\bibitem[CM18]{Combettes:2018b}
P.~L. Combettes and C.~L. M\"uller.
\newblock Perspective functions: {P}roximal calculus and applications in high-dimensional statistics.
\newblock {\em Journal of Mathematical Analysis and Applications}, 457(2):1283--1306, 2018.

\bibitem[Com18]{Combettes:2018}
P.~L. Combettes.
\newblock Perspective functions: Properties, constructions, and examples.
\newblock {\em Set-Valued and Variational Analysis}, 26(2):247--264, 2018.

\bibitem[DB16]{Diamond:2016}
S.~Diamond and A.~Boyd.
\newblock {CVXPY}: a {P}ython-embedded modeling language for convex optimization.
\newblock {\em Journal of Machine Learning Research}, 17(83):1--5, 2016.

\bibitem[DCB13]{Domahidi:2013}
A.~Domahidi, E.~Chu, and S.~Boyd.
\newblock {ECOS}: an {SOCP} solver for embedded systems.
\newblock In {\em European Control Conference (ECC)}, 2013.

\bibitem[DHL17]{Dunning:2017}
I.~Dunning, J.~Huchette, and M.~Lubin.
\newblock {JuMP}: a modeling language for mathematical optimization.
\newblock {\em SIAM Review}, 59(2):295--320, 2017.

\bibitem[GC26]{Goulart:2026}
P.~Goulart and Y.~Chen.
\newblock Clarabel: an interior-point solver for conic programs with quadratic objectives.
\newblock {\em Mathematical Programming Computation}, 2026.

\bibitem[GCG21]{Garstka:2021}
M.~Garstka, M.~Cannon, and P.~Goulart.
\newblock {COSMO}: a conic operator splitting method for convex conic problems.
\newblock {\em Journal of Optimization Theory and Applications}, 190(3):779--810, 2021.

\bibitem[GT56]{Goldman:1956}
A.~J. Goldman and A.~W. Tucker.
\newblock Theory of linear programming.
\newblock In H.~W. Kuhn and A.~W. Tucker, editors, {\em Linear Inequalities and Related Systems}, volume~38 of {\em Annals of Mathematics Studies}, pages 53--97. Princeton University Press, 1956.

\bibitem[LM79]{Lions:1979}
P.~Lions and B.~Mercier.
\newblock Splitting algorithms for the sum of two nonlinear operators.
\newblock {\em SIAM Journal on Numerical Analysis}, 16(6):964--979, 1979.

\bibitem[LSZ00]{Luo:2000}
Z.~Q. Luo, J.~F. Sturm, and S.~Zhang.
\newblock Conic convex programming and self-dual embedding.
\newblock {\em Optimization Methods and Software}, 14(3):169--218, 2000.

\bibitem[{MOS}26]{mosek}
{MOSEK ApS}.
\newblock {\em MOSEK Optimization Suite}, 2026.

\bibitem[NN94]{Nesterov:1994}
Y.~Nesterov and A.~Nemirovski.
\newblock {\em Interior-Point Polynomial Algorithms in Convex Programming}.
\newblock Society for Industrial and Applied Mathematics, 1994.

\bibitem[NT08]{Nemirovski:2008}
A.~Nemirovski and M.~J. Todd.
\newblock Interior-point methods for optimization.
\newblock {\em Acta Numerica}, 17:191--234, 2008.

\bibitem[OCPB16]{O'Donoghue:2016}
B.~O'Donoghue, E.~Chu, N.~Parikh, and S.~Boyd.
\newblock Conic optimization via operator splitting and homogeneous self-dual embedding.
\newblock {\em Journal of Optimization Theory and Applications}, 169(3):1042--1068, 2016.

\bibitem[O'D21]{O'Donoghue:2021}
B.~O'Donoghue.
\newblock Operator splitting for a homogeneous embedding of the linear complementarity problem.
\newblock {\em SIAM Journal on Optimization}, 31(3):1999--2023, 2021.

\bibitem[Roc70]{Rockafellar:1970}
R.~T. Rockafellar.
\newblock {\em Convex Analysis}.
\newblock Princeton University Press, USA, 1970.

\bibitem[SBG{\etalchar{+}}20]{Stellato:2020}
B.~Stellato, G.~Banjac, P.~Goulart, A.~Bemporad, and S.~Boyd.
\newblock {OSQP}: an operator splitting solver for quadratic programs.
\newblock {\em Mathematical Programming Computation}, 12(4):637--672, 2020.

\bibitem[YTM94]{Ye:1994}
Y.~Ye, M.~J. Todd, and S.~Mizuno.
\newblock An {$O(\sqrt{n}L)$}-iteration homogeneous and self-dual linear programming algorithm.
\newblock {\em Mathematics of Operations Research}, 19(1):53--67, 1994.

\bibitem[Zha04]{Zhang:2004}
S.~Zhang.
\newblock A new self-dual embedding method for convex programming.
\newblock {\em Journal of Global Optimization}, 29(4):479--496, 2004.

\end{thebibliography}

\end{document}